\documentclass[english]{amsart}

\usepackage{amsmath}
\usepackage{amssymb}
\usepackage{amsfonts}
\usepackage{amsthm}
\usepackage{mathrsfs}
\usepackage[all]{xy}
\usepackage[pdftex]{graphicx}
\usepackage{rotating}
\usepackage{color}
\usepackage{enumerate}
\usepackage{hyperref}

\newcommand{\gal}{\operatorname{Gal}}
\newcommand{\Hom}{\operatorname{Hom}}

\newcommand{\End}{\operatorname{End}}
\newcommand{\Frob}{\operatorname{Frob}}

\newcommand{\Ext}{\operatorname{Ext}}

\newcommand{\id}{\operatorname{id}}

\newcommand{\N}{\operatorname{N}}

\newcommand{\res}{\operatorname{Res}}

\newcommand{\isom}{\stackrel{\sim}{\rightarrow}}

\newcommand{\im}{\operatorname{im}}

\newcommand{\JH}{\operatorname{JH}}

\newcommand{\Fil}{\operatorname{Fil}}

\newcommand{\gr}{\operatorname{gr}}
\newcommand{\Rep}{\operatorname{Rep}}

\newcommand{\SL}{\mathrm{SL}}

\newcommand{\GL}{\mathrm{GL}}
\newcommand{\R}{\mathrm{R}}

\newcommand{\rad}{\mathrm{rad}}

\newcommand{\surj}{\twoheadrightarrow}
\newcommand{\inj}{\hookrightarrow}

\newcommand{\comment}[1]{}

\newcommand{\teich}[1]{\widetilde{#1}}
\newcommand{\tld}[1]{\widetilde{#1}}
\newcommand{\bA}{\mathbf{A}}
\newcommand{\bT}{\mathbf{T}}
\newcommand{\m}{\mathfrak{m}}

\newcommand{\Art}{\mathrm{Art}}

\newcommand{\un}[1]{\underline{#1}}
\newcommand{\into}{\hookrightarrow}
\newcommand{\La}{\Lambda}

\newcommand{\der}{\mathrm{der}}
\newcommand{\ra}{\rightarrow}

\newtheorem{thm}{Theorem}[section]
\newtheorem{lem}[thm]{Lemma}

\newtheorem{prop}[thm]{Proposition}
\newtheorem{corr}[thm]{Corollary}

\theoremstyle{remark}
\newtheorem{remark}[thm]{Remark}
\theoremstyle{definition}
\newtheorem{defin}[thm]{Definition}

\numberwithin{equation}{section}

\def\N{\mathbf N}
\def\Z{\mathbf Z}
\def\Q{\mathbf Q}

\def\F{\mathbf{F}}
\def\R{\mathbf{R}}

\def\G{\mathbf{G}}

\def\rarrow{\rightarrow}

\def\onto{\twoheadrightarrow}

\def\cO{\mathcal{O}}

\def\rbar{\overline{r}}
\def\rhobar{\overline{\rho}}
\def\psibar{\overline{\psi}}

\def\Fbar{\overline{\mathbf{F}}}

\DeclareMathAlphabet{\mathpzc}{OT1}{pzc}{m}{it}

\title{Multiplicity one at full congruence level}

\author{Daniel Le}
\address{University of Toronto,
40 St. George Street, 
Toronto, ON M5S 2E4}
\email{daniel.le@math.toronto.edu}

\author{Stefano Morra}
\address{Universit\'e Paris 8, Laboratoire d'Analyse, G\'eom\'etrie et Applications,  LAGA, Universit\'e Sorbonne Paris Nord, CNRS, UMR 7539,  F-93430, Villetaneuse, France}
\email{morra@math.univ-paris13.fr}

\author{Benjamin Schraen}
\address{Universit\'e Paris-Saclay, CNRS, Laboratoire de math\'ematiques d’Orsay, 91405, Orsay, France}
\email{benjamin.schraen@math.u-psud.fr}

\begin{document}

\begin{abstract}
Let $F$ be a totally real field in which $p$ is unramified.
Let $\rbar: G_F \rarrow \GL_2(\Fbar_p)$ be a modular Galois representation which satisfies the Taylor--Wiles hypotheses and is tamely ramified and generic at a place $v$ above $p$.
Let $\mathfrak{m}$ be the corresponding Hecke eigensystem.
We describe the $\mathfrak{m}$-torsion in the mod $p$ cohomology of Shimura curves with full congruence level at $v$ as a $\GL_2(k_v)$-representation.
In particular, it only depends on $\rbar|_{I_{F_v}}$ and its Jordan--H\"{o}lder factors appear with multiplicity one.
The main ingredients are a description of the submodule structure for generic $\GL_2(\F_q)$-projective envelopes and the multiplicity one results of \cite{EGS}.
\end{abstract}

\maketitle

\section{Introduction}

Fix a prime $p$ and a totally real field $F/\Q$.
Fix a modular Galois representation $\rbar:G_F \rarrow \GL_2(\Fbar_p)$ with corresponding Hecke eigensystem $\mathfrak{m}$.
Fix a place $v|p$ of $F$.
Mod $p$ local-global compatibility predicts that the $\mathfrak{m}$-torsion subspace, which we denote by $\pi$, in the mod $p$ cohomology of a Shimura curve with infinite level at $v$ realizes the mod $p$ Langlands correspondence for $\GL_2(F_v)$ (see \cite{BreuilICM}), generalizing the case of modular curves (\cite{Colmez,Emerton,Paskunas}).
The goal of the mod $p$ local Langlands program is then to describe $\pi$ in terms of the restriction to the decomposition group at $v$, $\rbar|_{G_v}$, though it is not even known whether $\pi$ depends only on $\rbar|_{G_v}$.
One of the major difficulties is that little is known about supersingular representations outside of the case of $\GL_2(\Q_p)$ (see \cite{AHHV}).

We now assume that $p$ is unramified in $F$ and that $\rbar|_{G_v}$ is $1$-generic (see Definition \ref{defin:generic}).
Let $K = \GL_2(\cO_v)$ and $I_1 \subset K$ be the usual pro-$p$ Iwahori subgroup.
\cite{BDJ} and \cite[Conjecture B.1]{Breuil} conjecturally describe the $K$-socle and $I_1$-invariants of $\pi$---in particular they should satisfy mod $p$ multiplicity one when the tame level is minimal (see \S\ref{sec:global}).
\cite{Gee} and \cite{EGS} later confirmed these conjectures.
\cite{Breuil} shows that such a $\pi$ (also satisfying other properties known for $\mathfrak{m}$-torsion in completed cohomology) must contain a member of a family of representations constructed in \cite{BP}.
If $f=1$, this family has one element, and produces the (one-to-one) mod $p$ Langlands correspondence for $\GL_2(\Q_p)$.
For $f>1$, each family is infinite (see \cite{Hu}), and so a na\"{i}ve one-to-one correspondence cannot exist.
Moreover, the $K$-socle and the $I_1$-invariants are not sufficient to specify a single mod $p$ $\GL_2(\Q_{p^f})$-representation when $f>1$.

However, \cite{EGS} proves a stronger multiplicity one result than what is used in the construction of \cite{BP}, namely a result for any lattice in a tame type with irreducible cosocle.
We strengthen this result in tame situations as follows (cf. Corollary \ref{corr:main}).
Let $K(1) \subset K$ be the kernel of the natural map $K \surj \GL_2(k_v)$. Assume that in the definition of $\pi$ we consider the cohomology of a Shimura curve with infinite level at $v$ and minimal tame level (see \S\ref{sec:global} for a precise statement).

\begin{thm}\label{thm:intro}
Suppose that $\rbar$ is $1$-generic and tamely ramified at $v$ and satisfies the Taylor--Wiles hypotheses.
Then the $\GL_2(k_v)$-representation $\pi^{K(1)}$ is isomorphic to the representation $D_0(\rbar|_{G_v})$ \emph{(}which depends only on $\rbar|_{I_v}$\emph{)} constructed in \cite{BP}.
In particular, its Jordan--H\"{o}lder constituents appear with multiplicity one.
\end{thm}

If the Jordan--H\"{o}lder constituents of a $\GL_2(k_v)$-representation appear with multiplicity one, we say that the representation is {\it multiplicity free}.

\begin{corr}
For $p>3$, there exists a supersingular $\GL_2(F_v)$-representation $\pi$ such that $\pi^{K(1)}$ is multiplicity free.
\end{corr}
\begin{remark}
We know of no purely local proof of this result.
\end{remark}
\begin{proof}
We can and do choose $\rbar$ such that $\rbar|_{G_v}$ is $1$-generic and irreducible by \cite[Corollary A.3]{GK}.
Then the $\GL_2(F_v)$-socle $\pi'$ of $\pi$ is supersingular (and irreducible) by \cite[Corollary 10.2.3]{EGS} and \cite[Theorem 1.5(i)]{BP}, and $\pi'^{K(1)} \subset \pi^{K(1)}$ is multiplicity free by Theorem \ref{thm:intro}.
\end{proof}

The theorem is obtained by combining results of \cite{EGS} with a description of the submodule structure of generic $\GL_2(k_v)$-projective envelopes (see Theorem \ref{thm:structure}).
Note that this theorem precludes infinitely many representations constructed in the proof of \cite[Theorem 4.17]{Hu} from appearing in completed cohomology.
It is not clear to the authors whether the results of \cite{Breuil,EGS} uniquely characterize $\pi$ when $\rbar$ is tamely ramified.

We now make a brief remark on the genesis of this paper.
The second and third authors arrived independently at a proof of Theorem \ref{thm:intro} (in an unreleased preprint) following a different argument, but related to the strategy presented here which was outlined in an unreleased preprint by the first author.
Relating the two approaches led to this collaboration.
After our paper had been written, we were notified that Hu and Wang also obtained a similar result independently \cite{Hu-Wang}.

We now give a brief overview of the paper.
In Section \ref{sec:ext}, we describe the extension graph, which simplifies the combinatorics of Serre weights.
Section \ref{sec:proj} is the technical heart of the paper, where we describe the submodule structure of generic $\GL_2(\F_q)$-projective envelopes.
In Section \ref{sec:local}, we use the results of Section \ref{sec:proj} to give two different characterizations of a construction of \cite{BP}.
Finally, in Section \ref{sec:global}, we derive our main result.

\subsection{Acknowledgments}
Many of the ideas in this article, especially the combinatorics of Section \ref{sec:ext} and proof of Proposition \ref{prop:ext}, came out of the joint work \cite{LLLM2} of the first two authors with Bao V. Le Hung and Brandon Levin.
We thank them heartily for their collaboration.
The first author thanks Florian Herzig for answering questions and providing references on modular representation theory, Yongquan Hu for answering a question about $K(1)$-invariants, and Matthew Emerton for numerous long and enlightening discussions about $p$-adic Langlands. The second and third authors were visitors at I.H.E.S. when they first worked on this topic.
The debt this article owes to the work of Christophe Breuil, Matthew Emerton, Toby Gee, Vytautas Pa\v{s}k\={u}nas, and David Savitt will be obvious to the reader.
The first author was supported by the National Science Foundation under agreement No.~DMS-1128155.

\subsection{Notation}
\label{sec:notation}
We introduce some notation that will be in force throughout.
If $F$ is any field, we write $\overline{F}$ for a separable closure of $F$ and $G_F:= \gal(\overline{F}/F)$ for the absolute Galois group of $F$.
If $F$ is a global field and $v$ is a place of $F$, we fix an embedding $\overline{F} \inj \overline{F}_v$, and we write $I_v\subset G_v$ to denote the inertia and decomposition subgroups at $v$ of $G_{F}$. We further write $\varpi_v\in F_v$ to denote an uniformizer. 
If $W_{F_v}\leq G_{F_v}$ denotes the Weil group of $G_{F_v}$ we normalize Artin's reciprocity map $\mathrm{Art}_{F_v}: F_v^\times\ra W_{F_v}^{\mathrm{ab}}$ in such a way that the geometric Frobenius elements are sent to uniformizers.

Throughout the paper, the place $v$ will divide $p$, and $F_v/\Q_p$ will be an unramified extension of degree $f$.
Let $q = p^f$.
We fix a coefficient field $\F$ which is a finite extension of $\F_q$.
Without further mention, all representations will be over $\F$.
We fix an embedding $\iota_0:\F_q \inj \F$.
The letters $i$ and $j$ will denote elements of $\Z/f$.
Let $\iota_i = \iota_0\circ \varphi^i$ be the $i$-th Frobenius twist of $\iota_0$.

Let $G$ be the algebraic group $\res_{\F_q/\F_p} \GL_2$.
Let $Z\subset T$ (resp. $Z_{\GL_2}\subset T_{\GL_2}$) be the center in the diagonal torus in $\res_{\F_q/\F_p} \GL_2$ (resp. in $\GL_2$).
Note that the choice of $\iota_0$ gives an isomorphism 
\begin{equation}\label{eqn:splitting}
T\times_{\F_p} \F \cong \prod_{i\in \Z/f} T_{{\GL_2}_{/\F}}.
\end{equation}
Thus, the Weyl group $W$ of $(G,T)$ (and sometimes the analogous version for $\SL_2$) is identified with $S_2^f$. Let $W_{\GL_2}$ be the Weyl group of $(\GL_2,T_{\GL_2})$, we denote by $w_0$ the non trivial element of $W_{\GL_2}$.

Let $X^*(T):=X^*(T\times_{\F_p} \Fbar_p)$ be the character group which is identified with $(\Z^{2})^{f}$ by (\ref{eqn:splitting}) and let $\varepsilon'_{i}\in X^*(T)$ correspond to the $f$-tuple which is $(1,0)$ in the $i$-th coordinate and $(0,0)$ otherwise.
Let $\eta = \sum_i \varepsilon'_i$. 
We denote by $C_0$ the base $p$-alcove in $X^*(T)\otimes_{\Z}\R$, i.e. the set of $\lambda\in X^*(T)\otimes_{\Z}\R$ such that $0<\langle \lambda+\eta,\alpha^\vee\rangle <p$ for all positive coroots $\alpha^\vee$.
(We define the positive coroots with respect to the Borel of upper triangular matrices in all embeddings.)
Let $X^0(T) \subset X^*(T)$ be the subgroup generated by $\iota_i \circ\det$ for $i\in\Z/f$.
We say that a weight $\mu$ is \emph{$p$-restricted} if $0\leq\langle \mu,\alpha^\vee\rangle <p$ for all positive coroots $\alpha^{\vee}$.
It is customary to write $X_1(T)$ for the set of $p$-restricted weights.

Let $G^{\der}=\res_{\F_q/\F_p} \SL_2$ and $T^{\der}$ be the diagonal torus. We write $\Lambda_W=X^*(T^{\der})$ for the weight lattice for $G^{\der}$  and  $\Lambda_R \subset \Lambda_W$ for the root lattice.

Note that the root lattice of $G$ is canonically isomorphic to $\La_R$, and we fix this identification from now on.
Note moreover that the restriction map induces a surjection $X^*(T)\twoheadrightarrow \Lambda_W$ with kernel $X^0(T)$.
Let $\varepsilon_i$ be the image of $\varepsilon'_i$ via the surjection $X^*(T)\twoheadrightarrow \Lambda_W$.

Let $\pi$ be the action of Frobenius on $X^*(T)$ so that, for instance, $\pi \varepsilon'_i =\varepsilon'_{i+1}$.

For a dominant character $\mu\in X^*(T)$ we write $V(\mu)$ for the Weyl module defined in \cite[II.2.13(1)]{JantzenBook}. It has a unique simple $G$-quotient $L(\mu)$. 
If $\mu = \sum_i \mu_i^{(i)}$, where $\mu_i\in X^*(T_{\GL_2})$ and $\mu_i^{(i)}=\iota_i\circ\mu_i$ for $i\in\Z/f$,
is $p$-restricted then $L(\mu) = \otimes_i L(\mu_i)^{(i)}$  by the Steinberg tensor product theorem as in \cite[Theorem 3.9]{Herzig} (as usual $L(\mu_i)^{(i)}$ denotes the $i$-th Frobenius twist of $L(\mu_i)$).
Let $\Gamma$ be the group $G(\F_p) \cong \GL_2(\F_q)$.
Let $F(\mu)$ be the $\Gamma$-representation $L(\mu)|_{\Gamma}$, which remains irreducible by \cite[A.1.3]{Herzig}.
Note that $F(\mu) \cong F(\lambda)$ if and only if $\mu \cong \lambda \mod{(p-\pi)X^0(T)}$.

Let $W^{\der}_a$ denote the affine Weyl group for $G^{\der}$ which is canonically isomorphic to the affine Weyl group $W_a$ of $G$. 
It is the semidirect product $\Lambda_R \rtimes W$.
Let $\widetilde{W}$ be the extended affine Weyl group of $G$ and $\widetilde{W}^{\der}$ be the extended affine Weyl group of $G^{\der}$.
They are defined as the semidirect product $X^*(T) \rtimes W$ and $\Lambda_W \rtimes W$ respectively. 
Note that we have a surjective morphism $\tld{W}\onto  \tld{W}^{\der}$ induced by $X^*(T)\onto \Lambda_W$.
If $\lambda\in \Lambda_R$ (resp. $\lambda\in X^*(T)$, resp. $\lambda\in \Lambda_W$) we write $t_\lambda$ for the image of $\lambda\in \Lambda_R$ (resp. $\lambda\in X^*(T)$, resp. $\lambda\in \Lambda_W$) under the usual embedding $\Lambda_R \inj W_a$ (resp. $X^*(T) \inj \widetilde{W}$, resp. $\Lambda_W \inj \widetilde{W}^{\der}$), i.e. $t_{\lambda}$ is the translation by $\lambda$. Note that we can extend the Frobenius action on the affine Weyl groups by declaring  $(\pi s)_j = s_{j+1}$ for $s \in W$.
There is a multiplication by $p$ isomorphism $\widetilde{W} \ra pX^*(T) \rtimes W$ sending $\widetilde{w} = t_\omega w\mapsto \widetilde{w}_p = t_{p\omega} w$.
For $\widetilde{w}\in \widetilde{W}$ we will use $\cdot$ to denote the $p$-dot action $\widetilde{w}\cdot \mu = \widetilde{w}_p(\mu+\eta) - \eta$.

Let $\Omega \subset \widetilde{W}$ be the stabilizer of $C_0$ under the $p$-dot action and $\Omega^{\der}$ its image under the map $\tld{W}\onto\tld{W}^{\der}$.
For instance, when $f=1$, the set $\Omega^{\der}$ is formed by the elements $\id$ and $(12)t_{-\varepsilon}$.
(Note that, in the notation of \cite{LLLM2}, $\Omega$ and $\Omega^{\der}$ would be denoted as $\tld{W}^+_1$ and $\tld{W}^{+,\der}_1$ respectively.)


\section{The extension graph}\label{sec:ext}

In this section, we describe what is called the extension graph in \cite[\S 2]{LLLM2} for $\GL_2$.
The modifications from $\GL_3$ are straightforward.

\begin{defin}
Let $S_e = \{\varepsilon_i\}_i$.
For $J \subset S_e$, let 
\begin{equation}
\label{eqn:rep}
\omega_J = \sum_{\omega\in J} \omega.
\end{equation}
\end{defin}

The inclusion ${\Lambda}_W \into \widetilde{W}^{\mathrm{der}}$ (resp.~$X^*({T}) \into \tld{W}$) induces an isomorphism $\iota^{\der}:\La_W/{\La}_R \stackrel{\sim}{\rightarrow} \widetilde{{W}}^{\der}/{W}^{\der}_a$ (resp.~$\iota:X^*({T})/{\Lambda}_R \stackrel{\sim}{\rightarrow} \tld{{W}}/{W}_a$).
Let $\mathcal{P}^{\mathrm{der}} \subset {\Lambda}_W \times \Omega^{\der}$ be the subset of pairs $(\omega,\tld{w})$ with $\iota^{\der}(-\pi^{-1}(\omega)+{\Lambda}_R) = \tld{w}{W}^{\mathrm{der}}_a$.

We similarly define $\mathcal{P} \subset X^*({T}) \times \Omega$.
Note that restriction gives a natural surjection $\mathcal{P}\onto \mathcal{P}^{\der}$.

The following lemma is easily checked.

\begin{lem}\label{lem:easy}
The map $(\omega,\tld{w})\mapsto \omega$ induces a bijection $\beta:\mathcal{P}^{\der}\stackrel{\sim}{\longrightarrow}\La_W$.
\end{lem}

Given $J\subseteq S_e$ we write $(\omega_J,\tld{w}_J)$ for the element of $\mathcal{P}^{\der}$ mapped to $\omega_J$ via $\beta$, with decomposition $\widetilde{w}_J = w_Jt_{-\pi^{-1}\omega_J}$ where $w_J\in W$.


Following \cite[Definition 2.1.2]{LLLM2} we have 
\begin{defin}\label{defn:regular}
We say that a weight $\lambda \in X_1(T)$ is {\it regular $p$-restricted} (or simply $p$-regular) if $0\leq \langle \lambda,\alpha^\vee \rangle <p-1$ for all positive roots $\alpha \in \Lambda_R$. 
We write $X_{\mathrm{reg}}(T)\subseteq X_1(T)$ for the set of regular $p$-restricted weights.
\end{defin}

Let $\mu$ be an element of $X_{\mathrm{reg}}(T)/(p-\pi)X^0({T})$. 
While we will often fix some lift of $\mu$ in $X^*({T})$, the constructions below will not depend on the choice of this lift.
We define a map
\begin{align}\label{eqn:tr}
\mathcal{P}^{\der} &\ra X^*({T})/(p-\pi)X^0({T}) \\
(\omega,\tld{w}) &\mapsto \tld{w}'\cdot(\mu-\eta+\omega'),\nonumber
\end{align}
where $(\omega',\tld{w}')\in \mathcal{P}$ is a lift of $(\omega,\tld{w})$.
The map (\ref{eqn:tr}) does not depend on the choice of lift. 
Then we define
\[\mathfrak{t}'_\mu: {\Lambda}_W \ra X^*({T})/(p-\pi)X^0({T})\]
to be the composition of $\beta^{-1}$ with (\ref{eqn:tr}).



Define $\Lambda^\mu_W$ to be the set
\[
\Lambda^\mu_W=\left\{
\omega \in \Lambda_W  :  0<\langle\omega+\mu,\alpha^\vee\rangle<p
\right\}.
\]
(where we take the image of $\mu$ in ${\Lambda}_W$).
Let $\mathfrak{t}_\mu$ be the restriction of $\mathfrak{t}_\mu'$ to ${\Lambda}_W^\mu$.

We establish some properties of $\mathfrak{t}_\mu$. 

\begin{prop}\label{prop:inj}
Let $\mu\in X_{\mathrm{reg}}(T)$. 
If $\omega\in \La_W^\mu$, then any lift to $X^*(T)$ of $\mathfrak{t}_\mu(\omega)$ is regular and $p$-restricted.
Moreover, the map $\mathfrak{t}_\mu$ is injective.
\end{prop}
\begin{proof}
The proof is analogous to that of \cite[Proposition 2.1.3]{LLLM2}.
\end{proof}

The following proposition gives symmetries of the extension graph.

\begin{prop}\label{prop:sym}
Let $\mu\in X_{\mathrm{reg}}(T)$. Let $\omega\in\La_W^\mu$ and let $\lambda-\eta$ be a lift of  $\mathfrak{t}_\mu(\omega)$ and $\beta^{-1}(\omega) = (\omega,\tld{w})$. 
Then 
\[\mathfrak{t}_\lambda(\nu) = \mathfrak{t}_\mu(w^{-1}(\nu)+\omega)\]
for $\nu\in {\La}_W^{\lambda}$, where $w \in {W}$ is the image of $\tld{w}$.
\end{prop}
\begin{proof}
This follows by a direct computation analogous to the one in the proof of \cite[Proposition 2.1.5]{LLLM2}.
\end{proof}

We now recall the definition of the depth of a weight.

\begin{defin}
\label{defin:deep}
Let $\lambda\in X^*(T)$ be a dominant weight and let $n\in \N$. 
We say that $\lambda$ lies $n$-deep in its alcove if for each positive coroot $\alpha^{\vee}$ there exists an integer $m_\alpha\in \Z$ such that $pm_{\alpha}+n<\langle \lambda+\eta, \alpha^{\vee}\rangle <p(m_\alpha+1)-n$.
\end{defin}

Note that Definition \ref{defin:deep} above is consistent with \cite[Definition 2.1.9]{LLLM2} and that $\lambda \in X_1(T)$ is $p$-regular if and only if it is $0$-deep.

\begin{defin}\label{defin:genchar}
Let $\mu = \sum_i \mu_i^{(i)} \in X^*(T)$ be a $p$-restricted weight.

We say that $\mu$ is \emph{generic} if $\mu-\eta$ (which lies in the closure of the alcove $C_0$) is $1$-deep.

An element of $\mu\in X^*({T})/(p-\pi)X^0({T})$ is generic if any lift of $\mu$ is. Note that a generic weight is $p$-regular.

\end{defin}

Following \cite{LLLM2}, we introduce the notion of adjacency in the extension graph.

\begin{defin}\label{defin:adj}
Two elements $\omega,\ \omega'\in \Lambda^{\mu}_W$ are said to be \emph{adjacent} if $\omega-\omega'\in\{\pm\varepsilon_j\}$ for some index $j$.
\end{defin}

We now justify the term ``extension graph''.
Recall that $\Gamma$ denotes the group $G(\F_p) \cong \GL_2(\F_q)$.
A \emph{Serre weight} is an absolutely irreducible representation of $\Gamma$ over an $\F$-vector space.
Each Serre weight is obtained by restriction to $\Gamma$ from an irreducible algebraic representation of ${G}$ of highest weight $\lambda\in X_1({T})$, and this process gives a bijection between from $X_1({T})/(p-\pi)X^0({T})$ to the set of Serre weights of $\Gamma$ (as described in \cite[Theorem 3.10]{Herzig}). 
As we mentioned in \S \ref{sec:notation}, given $\lambda\in X_1({T})$ we write $F(\lambda)$ for the Serre weight corresponding to $\lambda$.
We say that a Serre weight $F$ is $p$-\emph{regular} if $F\cong F(\lambda)$ where $\lambda\in X_1(\un{T})$ is regular $p$-restricted (cf.~Definition \ref{defn:regular}).
Given $\mu\in{C}_0$ and $\omega\in \un{\La}_W^\mu$, we get a corresponding $p$-regular Serre weight $F(\mathfrak{t}_\mu(\omega))$.
One can prove (cf.~\cite[Propositions 2.1.3 and 2.1.4]{LLLM2}) that $F(\mathfrak{t}_\mu(-))$ induces a bijection between the set ${\Lambda}_W^\mu$ and the set of $p$-regular Serre weights of $\Gamma$ with the same central character as $F(\mu-\eta)$.

\begin{prop}\label{prop:extgraph}
Let $\mu\in C_0$. 
Let $\omega, \omega' \in \Lambda_W^\mu$ such that $\lambda-\eta:=\mathfrak{t}_\mu(\omega)$ and $\lambda'-\eta:=\mathfrak{t}_\mu(\omega')$ are generic.
Then 
\[\dim \Ext_\Gamma^1(F(\lambda-\eta),F(\lambda'-\eta)) = \dim \Ext_\Gamma^1(F(\lambda'-\eta),F(\lambda-\eta)) \leq 1\]
with equality if and only if $\omega$ and $\omega'$ are adjacent in the graph $\Lambda_W^{\mu}$.
\end{prop}
\begin{proof}
By Proposition \ref{prop:sym}, we can assume without loss of generality that $\omega = 0$.
Then the extensions of $\sigma:=F(\mu-\eta)$ are given by the first layer of the cosocle filtration of the projective envelope of $\sigma$.
The proposition now follows from Propositions \ref{prop:tensor} and \ref{prop:Wdecomp} (which do not depend on this proposition).
\end{proof}

We next show that a set of modular Serre weights forms a hypercube in the extension graph.  

We write $\mathscr{W}$ for the set of Serre weights.
By the discussion preceeding Proposition \ref{prop:extgraph}, this is in bijection with the image of $X_1(T)$ in $X^*(T)/(p-\pi)X^0(T)$.
Write $\mathscr{W}_{\mathrm{reg}}$ for the set of regular Serre weights which is in bijection with the image of $X_{\mathrm{reg}}(T)$ in $\mathscr{W}$. We have a bijection $\mathcal{R}:X^*(T)\ra X^*(T)$  (also called \emph{Herzig reflection}) defined by $\lambda\mapsto w_0t_{-\eta}\cdot \lambda$. It induces a bijection $\mathcal{R}:\mathscr{W}_{\mathrm{reg}}\ra \mathscr{W}_{\mathrm{reg}}$.

\begin{defin}
\label{defn:generic}
A Serre weight $F$ is said to be $n$-deep (resp.~generic) if we can write $F\cong F(\lambda)$ for a weight $\lambda\in X_1(T)$ which is $n$-deep (resp.~generic).
\end{defin}

For $s\in W$ and a character $\mu\in X^*(T)$, we denote the corresponding Deligne--Lusztig representation as in \cite[Lemma 4.2]{Herzig} by $R_s(\mu)$. 

We always assume that $R_s(\mu)$ is defined over $W(\F)$, the ring of Witt vectors of $\F$.
Given a Deligne--Lusztig representation $R_s(\mu)$ as above, we write $\JH(\overline{R}_s(\mu))$ to denote the set of Jordan--H\"older constituents of the mod $p$ reduction of a $\Gamma$-invariant $W(\F)$-lattice inside $R_s(\mu)$.

It is easy to see that if $\mu-\eta$ is $n$-deep 
then any weight $F(\lambda-\eta)\in\JH(\overline{R}_s(\mu))$ is $n-1$-deep. 
In particular, if $\mu-\eta$ is $1$-deep, then $\#\JH(\overline{R}_s(\mu))=2^f$ and all the Jordan-H\"older constituents in $\JH(\overline{R}_s(\mu))$ are $0$-deep.
Following \cite[\S 9.1]{GHS} an $L$-parameter for $G$ is, in our context, equivalent to a continuous homomorphism $I_{F_v}\rightarrow \GL_2(\F)$ which extends to $G_{F_v}$. Given an inertial $L$-parameter $\tau$ we can associate a Deligne--Lusztig representation $V_{\phi}(\tau)$ following \cite[Proposition 9.2.1]{GHS}. We define the set $W^?(\tau)$ as
\[
W^?(\tau)=\{ \mathcal{R}(F),\ F\in\JH(\overline{V_{\phi}(\tau)})\}
\]
where, similarly as above, the notation $\JH(\overline{V_{\phi}(\tau)})$ stands for the set of Jordan--H\"older constituents of the mod $p$ reduction of a $\Gamma$-invariant $W(\F)$-lattice inside $V_{\phi}(\tau)$.

\begin{prop}\label{prop:weights}
Suppose that $\tau$ is an inertial $L$-parameter such that $V_{\phi}(\tau) = R_s(\mu)$. Assume that $\mu-\eta$ is $1$-deep.
Then $W^?(\tau) = F(\mathfrak{t}_\mu(\{s\omega_J:J\subset S_e\}))$.
\end{prop}
\begin{proof}
The obvious crystalline lifts, in the sense of \cite[\S 7.1]{GHS}, have Hodge--Tate weights $wt_{s\pi \omega - p\omega}(\mu)$, where $wt_{-\omega}$ ranges over all elements of $\Omega$.
Noting that 
\begin{align*}
wt_{s \omega - p\pi^{-1}\omega}(\mu) - \eta &\equiv wt_{ -\pi^{-1}\omega}\cdot(\mu+s\omega-\eta) \text{ mod } (p-\pi)X^0(T)
\\ &\equiv\mathfrak{t}_\mu(s\omega) \text{ mod } (p-\pi)X^0(T)
\end{align*}
and that the image of $\{\omega\mid wt_{-\omega} \in \Omega\}$ in $\Lambda_W$ is $\{\omega_J\mid J\subset S_e\}$, we have that, in the notation of \cite{GHS}, $W_{\mathrm{obv}}(\tau) = F(\mathfrak{t}_\mu(\{s\omega_J:J\subset S_e\}))$.
Finally, $W^?(\tau) = W_{\mathrm{obv}}(\tau)$ (see \cite[\S 4.2]{Gee11}).
\end{proof}

\section{Generic $\GL_2(\F_q)$-projective envelopes} \label{sec:proj}

In this section, we describe the submodule structure of \emph{generic} $\GL_2(\F_q)$-projective envelopes, i.e.~$\GL_2(\F_q)$-projective envelopes of  Serre weights which are generic in the sense of Definition \ref{defn:generic}.

\vspace{2mm}

Recall that $\Gamma$ is the group $G(\F_p) \cong \GL_2(\F_q)$ and if $R$ is a $\Gamma$-representation, we write $R^{(i)}$ to denote its $i$-th Frobenius twist. In what follows we set $\mu = \sum_{i=0}^{f-1} \mu_i^{(i)}\in X^*(T)$ where $\mu_i = (a_i,b_i)\in\Z^2$.

Assume that $\mu -\eta$ is dominant. If we write $\mu-\eta=\sum_{i=0}^{f-1}r_i\varepsilon'_i+\sum_{i=0}^{f-1}d_i(1,1)^{(i)}$ Breuil and Paskunas define a $\Gamma$-representation $(R_{(r_i)_i})\otimes {\det}^{\sum_ip^id_i}$ in \cite{BP},\S3. We define $R_\mu$ to be the dual of the representation $(R_{(r_i)_i})\otimes {\det}^{\sum_ip^id_i}$.

The following known theorem gives a coarse description of generic $\Gamma$-projective envelopes.
\begin{thm}\label{thm:projenv}
Assume that  $1\leq a_i - b_i \leq p-1$ for all $i$. Then $R_\mu = \otimes_{i=0}^{f-1} R_{\mu_i}^{(i)}$, where 
\begin{enumerate}
\item $R_{\mu_i}$ is a $\Gamma$-representation with a filtration $\Fil^0 R_{\mu_i} = R_{\mu_i}$, $\Fil^1 R_{\mu_i} \cong V(w_0t_{(-1,1)} \cdot (\mu_i-\varepsilon'_0))$, $\Fil^2 R_{\mu_i} \cong F(\mu_i-\varepsilon'_0)$, and $\Fil^3 R_{\mu_i} = 0$, and 
\item $\gr^0R_{\mu_i}$ and $\gr^2R_{\mu_i}$ are isomorphic to $F(\mu_i-\varepsilon'_0)$ and $\gr^1 R_{\mu_i}$ is isomorphic to $F(w_0t_{-\varepsilon'_i} \cdot (\mu_i-\varepsilon'_0))\otimes F(\varepsilon'_0)^{(1)}$.
\end{enumerate}
Moreover, if there exists an index $i$ such that $a_i-b_i>1$ then $R_{\mu}$ is a projective (and injective) envelope of the weight $F(\mu-\eta)$. Else, if $a_i-b_i=1$ for all $i$, then the representation $R_\mu$ is isomorphic to the direct sum of the projective (and injective) envelope of the weight $F(\mu-\eta)$ and a twist of the Steinberg representation.
\end{thm}
\begin{proof}
See \cite[\S 3, Lemmas 3.4, 3.5]{BP}. 
\end{proof}

The filtrations on $R_{\mu_i}$ induce a tensor multifiltration on $R_\mu$.
More precisely, the set $\{0,1,2\}^f$ has a partial order so that $(k_i)_i = \mathbf{k} \leq \mathbf{k}' = (k'_i)_i$ if $k_i\leq k'_i$ for all $i\in \Z/f$.
We write $\mathbf{k} < \mathbf{k}'$ if $\mathbf{k}\leq \mathbf{k}'$ and $\mathbf{k}\neq \mathbf{k}'$.
For $\mathbf{k} = (k_i)_i \in \{0,1,2\}^f$, let $\Fil^{\mathbf{k}}R_\mu := \otimes_i \Fil^{k_{i+1}}R_{\mu_i}^{(i)}$.
Then $\Fil^{\mathbf{k}'}R_\mu \subsetneq \Fil^{\mathbf{k}}R_\mu$ if and only if $\mathbf{k} < \mathbf{k}'$.
Let $\Fil^{>\mathbf{k}}R_\mu = \sum_{\mathbf{k}< \mathbf{k}'} \Fil^{\mathbf{k}'}R_\mu$.
Let $\gr^{\mathbf{k}} R_\mu = \Fil^{\mathbf{k}}R_\mu/\Fil^{>\mathbf{k}}R_\mu$.
To ease notation, we will also denote $\gr^{\mathbf{k}} R_\mu$ by $W_{\mathbf{k}}$.
For $\mathbf{k} = (k_i)_i \in \{0,1,2\}^f$, let $|\mathbf{k}| = k = \sum_i k_i$.
There is also the tensor filtration $\Fil^k_\otimes R_\mu = \sum_{|\mathbf{k}| = k} \Fil^{\mathbf{k}} R_\mu$.
Note in particular that for all $\mathbf{k} \in \{0,1,2\}^f$ we have a natural surjection $R_\mu/\Fil^{>\mathbf{k}}R_\mu\surj R_\mu/\Fil^{k+1}_\otimes R_\mu$ whose restriction to $W_{\mathbf{k}}\subseteq R_\mu/\Fil^{>\mathbf{k}}R_\mu$ is injective.

\begin{prop} \label{prop:tensor}
$\gr^k_{\otimes} R_\mu = \oplus_{|\mathbf{k}| = k} W_{\mathbf{k}}$.
\end{prop}
\begin{proof}
This follows from general facts about tensor products of filtered objects.
\end{proof}
To describe the representations $W_{\mathbf{k}}$, we will need the following translation principles.   

\begin{prop}\label{prop:translation}
Let $\lambda-\eta,\ \omega\in X^*(T)$ be dominant weights in alcove $C_0$. 
Assume that $L(\omega)|_T$ is multiplicity free and that for all weights $\nu$ of $L(\omega)|_T$, $\lambda-\eta+\nu$ is still in alcove $C_0$.
Then we have an isomorphism
\begin{enumerate}
\item $F(\lambda-\eta) \otimes F(\omega) \cong \oplus_{\nu \in\JH(L(\omega)|_T)} F(\lambda-\eta+\nu).$
\end{enumerate}
Assume moreover that $\langle \lambda-\eta,\alpha^\vee\rangle>0$ for at least one positive coroot $\alpha^\vee$.
Then we have an isomorphism
\begin{enumerate}\setcounter{enumi}{1}
\item $R_{\lambda} \otimes F(\omega) \cong \oplus_{\nu \in\JH(L(\omega)|_T)}  P_{F(\lambda-\eta +\nu)}$
\end{enumerate}
where we have written $P_{F(\lambda-\eta +\nu)}$ to denote a projective envelope of the Serre weight $F(\lambda-\eta +\nu)$. (In particular, $P_{F(\lambda-\eta +\nu)}\cong R_{\lambda +\nu}$ if $F(\lambda-\eta +\nu)$ is not a character.)
\end{prop}
\begin{remark}
In the statement of the Proposition \ref{prop:translation} assume that $\langle \omega,\alpha^{\vee}\rangle \leq 1$ for all positive coroots $\alpha^{\vee}$. It is then easy to check that the proposition applies with $\omega = \varepsilon'_i$ for all $i$ as soon as $\lambda-\eta$ is $1$-deep.
\end{remark}
\begin{proof}
We prove the analogous results for $G^{\der}$.
By \cite[Lemma 5.1(i)]{Pil} we have a $G^{\der}$-decomposition $L(\lambda) \otimes L(\omega) \cong \oplus_{\nu \in\JH(L(\omega)|_T)} L(\lambda+\nu)$ and the first statement for $G^{\der}$ follows by restriction to the finite group $G^{\der}(\F_p)$.

As for the second statement, we need to recall some standard facts about injective envelopes of Frobenius kernels.
Let $T_{\SL_2}$ be the standard torus of ${\SL_2}_{/\F_p}$. 
For any $r\geq 1$ we let $(\SL_2)_r$ denote the $r$-th Frobenius kernel of $\SL_2$ and, for any weight $\lambda\in X_r(T_{\SL_2})$ we write $Q_r(\lambda)$ for the injective envelope of $L(\lambda)|_{(\SL_2)_rT_{\SL_2}}$. 
Under our assumption on $p$ the $(\SL_2)_rT_{\SL_2}$-module $Q_r(\lambda)$ has a unique $\SL_2$-module structure, as well as a $\SL_2$-equivariant decomposition:
\begin{equation}
\label{eq:dec:inj}
Q_r(\lambda)\cong \otimes_{i=0}^{r-1} Q_1(\lambda_i)^{(i)}
\end{equation}
if $\lambda$ decomposes as $\lambda=\sum_{i=0}^{r-1}p^i\lambda_i$ with each $\lambda_i\in X^*(T_{\SL_2})$ being $p$-restricted.

Assume now that $\omega\in X_r(T_{\SL_2})$ is such that $L(\omega)$ is multiplicity free and $\lambda+\nu$ lies in the same alcove as $\lambda$ for any weight $\nu \in\JH(L(\omega)|_T)$. By \cite[Lemma 5.1(ii)]{Pil} and \ref{eq:dec:inj} we have a decomposition
\begin{equation}
\label{eq:trans:inj}
Q_r(\lambda)\otimes L(\omega)\cong \oplus_{\nu \in\JH(L(\omega)|_T)} Q_r(\lambda+\nu)\cong 
\oplus_{\nu \in\JH(L(\omega)|_T)}\otimes_{i=0}^{r-1} Q_1(\lambda_i+\nu_i)^{(i)}
\end{equation}
where we have written $\nu=\sum_{i=0}^{r-1}p^i\nu_i$ with $\nu_i\in X_1(T_{\SL_2})$ for all $\nu \in\JH(L(\omega)|_T)$.

The second statement of the Proposition  for $G^{\der}$ follows now from \ref{eq:trans:inj}.

The statements for $G$ are now deduced from the previous results on $G^{\der}$ by a formal argument, cf.  for instance \cite{LLLM2}, Theorem 4.1.3.
 \end{proof}

From now on 
we assume that $\mu-\eta$ is $1$-deep. In particular $R_{\mu}$ is the projective envelope of the weight $F(\mu-\eta)$.


\begin{defin}
Let $S = \{\pm \varepsilon_i\}_i$, and $\mathfrak{J}$ be the set of subsets of $S$.
For $J\in \mathfrak{J}$ define $\omega_J := \sum_{\omega \in J} \omega\in \Lambda_W$ and  $\sigma_J := F(\mathfrak{t}_\mu(\omega_J))$. 
Finally, let $\mathbf{k}(J) := (k_i(J))_i\in\{0,1,2\}^f$ where $k_{i+1}(J) := \#\{\pm \varepsilon_i\} \cap J$, and let $k(J) := \# J = |\mathbf{k}(J)|$.
\end{defin}

The following key multiplicity one result allows one to give a reasonable description of the submodule structure of generic $\Gamma$-projective envelopes.

\begin{prop}\label{prop:Wdecomp}
Let $\mathbf{k}\in\{0,1,2\}^f$. Then
$W_{\mathbf{k}} \cong \underset{J\in \mathfrak{J},\ \mathbf{k}(J) = \mathbf{k}}{\bigoplus} \sigma_J$. Moreover, this sum is multiplicity free.
\end{prop}
\begin{proof}
By definition and Theorem \ref{thm:projenv}(2) we have $W_{\mathbf{k}} \cong \otimes_i (F(\lambda_i-\varepsilon'_0)\otimes F(\nu_i))^{(i)}$ where $\lambda_i - \varepsilon'_0 = w_0 t_{-\varepsilon'_0} \cdot (\mu_i - \varepsilon'_0)$ if $k_{i+1} = 1$ and $\lambda_{i} - \varepsilon'_0 =\mu_i-\varepsilon'_0$ otherwise, and $\nu_i = \varepsilon'_0$ if $k_i = 1$ and $\nu_i =0$ otherwise. Note that $\lambda_i - \varepsilon'_0$ is $n$-deep in its alcove $C_0$ if and only if $\mu_i-\varepsilon'_0$ is $n$-deep in its alcove.
By Proposition \ref{prop:translation}, $F(\lambda_i-\varepsilon'_0)\otimes F(\varepsilon'_0) \cong F(\lambda_i) \oplus F(\lambda_i+(-1,1))$.
In particular, $W_{\mathbf{k}}$ is semisimple and of length $2^\delta$, where $\delta = \#\{i:k_i = 1\}$.

Suppose that $J \in \mathfrak{J}$ such that $\mathbf{k}(J) = \mathbf{k}$.
Then $\mathfrak{t}_\mu(\omega_J)_i = (\lambda_i-\varepsilon'_0+\omega'_i)$ where $\omega'_i = 0$ if $k_i\neq 1$ and is $\varepsilon'_0$ or $w_0\varepsilon'_0$ otherwise.
By the last paragraph, there is an inclusion $\sigma_J \inj W_{\mathbf{k}}$.
One easily checks that $\#\{J\in\mathfrak{J}:\mathbf{k}(J) = \mathbf{k}\}=\#\{\omega_J:\mathbf{k}(J) = \mathbf{k}\}=2^\delta$.
Since $\#\{\omega_J:\mathbf{k}(J) = \mathbf{k}\} = \#\{\sigma_J:\mathbf{k}(J) = \mathbf{k}\}$ by Proposition \ref{prop:inj} and $W_{\mathbf{k}}$ is semisimple and of length $2^\delta$, we are done.
\end{proof}

By abuse of notation, $\sigma_J$ will often denote the $\sigma_J$-isotypic component of $W_{\mathbf{k}(J)}$, which is isomorphic to $\sigma_J$ by Proposition \ref{prop:Wdecomp}.

In what follows we fix $\mathbf{k} \in \{0,1,2\}^f$ and let $k = |\mathbf{k}|$.
Let
\[W_{\mathbf{k},\mathbf{k}+1}:=\Fil^{\mathbf{k}} R_\mu/\big(\Fil^{k+2}_{\otimes} R_\mu \cap \Fil^{\mathbf{k}} R_\mu \big) \subset \Fil^k_{\otimes} R_\mu/\Fil^{k+2}_{\otimes} R_\mu.\]

The module $W_{\mathbf{k},\mathbf{k}+1}$ 
is endowed with the induced filtration from $\Fil^k_{\otimes} R_\mu/\Fil^{k+2}_{\otimes}$.
This is a two step filtration with associated graded pieces described as follows.
We have $\gr^k W_{\mathbf{k},\mathbf{k}+1} = W_{\mathbf{k}}$ and $\gr^{k+1} W_{\mathbf{k},\mathbf{k}+1} = \oplus_{\mathbf{k}'} W_{\mathbf{k}'}$ where the direct sum ranges over the elements $\mathbf{k}'\in\{0,1,2\}^f$ satisfying $\mathbf{k} \leq \mathbf{k}'$ and $k'-k = 1$.
We have the following refinement of Proposition \ref{prop:Wdecomp}.

\begin{lem}\label{lem:Wdecomp}
Keep the previous hypotheses and notation. 
The graded piece
\[\gr^{k+1} W_{\mathbf{k},\mathbf{k}+1}\subset \gr^{k+1}_{\otimes}R_\mu\]
is multiplicity free.
\end{lem}
\begin{proof}
Suppose that $\sigma\in \JH(\gr^{k+1} W_{\mathbf{k},\mathbf{k}+1})$ is a constituent appearing with multiplicity. By Proposition \ref{prop:Wdecomp}, we deduce the existence of $J_1, J_2\in \mathfrak{J}$ with $\mathbf{k}(J_1)\neq \mathbf{k}(J_2)$, $\sigma_{J_1}\cong\sigma\cong\sigma_{J_2}$, and $\mathbf{k}(J_1),\ \mathbf{k}(J_2)$ are of the form $\mathbf{k}'$ above. In what follows, we write $(k_{1,i})_i= \mathbf{k}_1:=\mathbf{k}(J_1)$ and similarly  $\mathbf{k}_2:=\mathbf{k}(J_2)$.
Let $j_1, j_2\in\{0,\dots, f-1\}$ be the unique elements such that $k_{j_1+1} +1 = k_{1,j_1+1}$ and $k_{j_2+1} +1 = k_{2,j_2+1}$.
Then $j_1 \neq j_2$, and hence $k_{j_2+1} +1 = k_{1,j_1+1}$, from which we see that the $j_1$ component of $\omega_{J_1}$ and $\omega_{J_2}$ must differ.
By Proposition \ref{prop:inj}, we conclude that $\sigma_{J_1}\not\cong\sigma_{J_2}$, a contradiction.
\end{proof}

Let now $\mathbf{k}'\in \{0,1,2\}^f$ be as above and let $j\in \Z/f$ be such that $k_{i+1} = k'_{i+1}$ for $i\neq j$ and $k_{j+1} +1 = k'_{j+1}$.  We define 
\[W_{\mathbf{k},\mathbf{k}'}:=\otimes_{i\neq j} \gr^{k_{i+1}} R_{\mu_i}^{(i)} \otimes (\Fil^{k_{j+1}}R_{\mu_j}/\Fil^{k_{j+1}+2}R_{\mu_j})^{(j)},\]
which is a quotient of $W_{\mathbf{k},\mathbf{k}+1}$.
We endow $W_{\mathbf{k},\mathbf{k}'}$ with the induced quotient filtration from $W_{\mathbf{k},\mathbf{k}+1}$; it is a two step filtration with graded pieces $\gr^k W_{\mathbf{k},\mathbf{k}'} = W_{\mathbf{k}}$ and $\gr^{k+1} W_{\mathbf{k},\mathbf{k}'} = W_{\mathbf{k}'}$.


\begin{prop}\label{prop:ext}
Suppose that $J\subset J'$ and $\#J'\setminus J = 1$.
Let $\mathbf{k} = \mathbf{k}(J)$ and $\mathbf{k}' = \mathbf{k}(J')$.
Then there is a subquotient of $W_{\mathbf{k},\mathbf{k}'}$ which is the unique up to isomorphism nontrivial extension of $\sigma_J$ by $\sigma_{J'}$.

\end{prop}
\begin{proof}
Suppose that $J' \setminus J \subset \{\pm \varepsilon_j\}$ and that $k_{j+1} = 0$ (resp. $k_{j+1} = 1$).
It suffices to show that $\sigma_{J'}$ (resp. $\sigma_J$) is not in the cosocle (resp. the socle) of $W_{\mathbf{k},\mathbf{k}'}$.
Indeed, this would show that the image of the extension $W_{\mathbf{k},\mathbf{k}'}$ under the map (canonically defined up to scalar) $\Ext^1_\Gamma(W_{\mathbf{k}},W_{\mathbf{k'}}) \rarrow \Ext^1_\Gamma(W_{\mathbf{k}},\sigma_{J'})$ (resp. $\Ext^1_\Gamma(W_{\mathbf{k}},W_{\mathbf{k'}}) \rarrow \Ext^1_\Gamma(\sigma_J,W_{\mathbf{k}'})$) is nonzero.
Since by Proposition \ref{prop:extgraph}, the map (canonically defined up to scalar) $\Ext^1_\Gamma(W_{\mathbf{k}},\sigma_{J'}) \rarrow \Ext^1_\Gamma(\sigma_J,\sigma_{J'})$ (resp. $\Ext^1_\Gamma(\sigma_J,W_{\mathbf{k}'}) \rarrow \Ext^1_\Gamma(\sigma_J,\sigma_{J'}))$ is an isomorphism, we would be done.
We show the following: if $k_{j+1} = 0$ (resp. $k_{j+1} =1$), then the cosocle (resp. the socle) of $W_{\mathbf{k},\mathbf{k}'}$ is isomorphic to $W_{\mathbf{k}}$ (resp. $W_{\mathbf{k}'}$).

Assume that $k_{j+1} = 0$.
We freely use the notation in the proof of Proposition \ref{prop:Wdecomp}.
Recall that $W_{\mathbf{k}} \cong \otimes_i (F(\lambda_i-\varepsilon'_0)\otimes F(\nu_i))^{(i)}$, which is semisimple.
There is a surjection $\otimes_i (R_{\lambda_i}\otimes F(\nu_i))^{(i)} \surj W_{\mathbf{k},\mathbf{k}'}$.
Noting that $\lambda_i-\varepsilon'_0$ is $1$-deep for all $i$ we see that Proposition \ref{prop:translation} applies and hence the latter surjection is actually the projective envelope of the semisimple representation $W_{\mathbf{k}}$.
We conclude that the cosocle of $W_{\mathbf{k},\mathbf{k}'}$ is $W_{\mathbf{k}}$, as desired. 

If $k_{j+1} = 1$, one makes the dual argument using the injection $W_{\mathbf{k}'} \inj \otimes_i (R_{\lambda_i}\otimes F(\nu_i))^{(i)}$.
\end{proof}

Fix $J\in \mathfrak{J}$.
Recall that by Proposition \ref{prop:Wdecomp}, there is a unique submodule of $W_{\mathbf{k}(J)}\subset R_\mu/\Fil^{>\mathbf{k}(J)} R_\mu$ isomorphic to $\sigma_J$.
Let us write $\mathbf{k}:=\mathbf{k}(J)$ and $k:=|\mathbf{k}|$ in what follows.
Let $P_{\sigma_J}$ be a projective envelope of $\sigma_J$.
Then $\Hom_\Gamma(P_{\sigma_J},\gr^{k+1}_\otimes R_\mu) \cong \Hom_\Gamma(\sigma_J,\gr^{k+1}_\otimes R_\mu) = 0$ by Proposition \ref{prop:Wdecomp} and the fact that $\sigma_J \cong \sigma_{J'}$ implies that $\omega_J = \omega_{J'}$ by Proposition \ref{prop:inj}, which implies that $\# J \equiv \# J' \mod 2$.
Then since $P_{\sigma_J}$ is projective, the natural map 
\[\Hom_\Gamma(P_{\sigma_J},\Fil^k_\otimes R_\mu/\Fil^{k+2}_\otimes R_\mu) \ra \Hom_\Gamma(P_{\sigma_J},\gr^k_\otimes R_\mu)\]
is an isomorphism between vector spaces of dimension $1$.
Thus a fixed morphism $P_{\sigma_J} \surj \sigma_J \subset W_{\mathbf{k}}\subset \gr^k_\otimes R_\mu$ (unique up to scalar), uniquely lifts to a morphism $\psibar_J:P_{\sigma_J} \ra \Fil^k_\otimes R_\mu/\Fil^{k+2}_\otimes R_\mu$.
Note that since $\sigma_J \subset W_{\mathbf{k}}$, we could also take a lift of $P_{\sigma_J} \surj \sigma_J \subset W_{\mathbf{k}}$ in $\Hom_\Gamma(P_{\sigma_J}, W_{\mathbf{k},\mathbf{k}+1})$, which must coincide with $\psibar_J$ by uniqueness.
We conclude that the image of $\psibar_J$ lies in $W_{\mathbf{k},\mathbf{k}+1}$.
Let $\overline{V}_J$ be the image of $\psibar_J$, which obtains a filtration from $W_{\mathbf{k}, \mathbf{k}+1}$.
The following describes the structure of $\overline{V}_J$.

\begin{prop}\label{prop:Vbardecomp}
We have that $\gr^k_\otimes \overline{V}_J = \sigma_J$ and $\gr^{k+1}_\otimes \overline{V}_J = \oplus_{J'} \sigma_{J'}$ where the sum runs over $J'$ such that $J\subset J'$ and $\# J' - \# J = 1$.
\end{prop}
\begin{proof}
Since $\overline{V}_J$ has irreducible cosocle isomorphic to $\sigma_J$ and $\sigma_J \subset \gr^k_\otimes \overline{V}_J$, $\gr^k_\otimes \overline{V}_J = \sigma_J$.
By Proposition \ref{prop:ext}, for every $J'$ as in the statement of the theorem there is a subquotient $\sigma_{J,J'}$ of $W_{\mathbf{k}, \mathbf{k}+1}$ which is a nontrivial extension of $\sigma_J$ by $\sigma_{J'}$. Consequently there exists a non zero map $\overline{\psi}_J':\,P_{\sigma_J}\rightarrow W_{\mathbf{k},\mathbf{k}'}$ whose image contains $\sigma_{J,J'}$.
By uniqueness, the composition of $\overline{\psi}_J$ with projection to $W_{\mathbf{k},\mathbf{k}'}$ is $\overline{\psi}_J'$, and therefore $\sigma_{J,J'}$ is a quotient of $\overline{V}_J$.
We see that $\oplus_{J'} \sigma_{J'} \subset \gr^{k+1}_\otimes\overline{V}_J$.

Since $\gr^{k+1}_\otimes W_{\mathbf{k}, \mathbf{k}+1}$ is multiplicity free by Lemma \ref{lem:Wdecomp}, it suffices to show that if $\sigma_{J'}\subset \gr^{k+1}_\otimes W_{\mathbf{k}, \mathbf{k}+1}$ is a Jordan--H\"older factor of $\gr^{k+1}_\otimes\overline{V}_J$ then $J'$ has the above form.
Since $\overline{V}_J$ has Loewy length two and cosocle isomorphic to $\sigma_J$, if $\sigma_{J'}$ is a Jordan--H\"older factor of $\gr^{k+1}_\otimes\overline{V}_J$, $\overline{V}_J$ must have as a quotient a nontrivial extension of $\sigma_J$ by $\sigma_{J'}$.
Hence $\omega_{J'} - \omega_{J} = \pm\varepsilon_j$ for some $j$ by Proposition \ref{prop:extgraph}.
Since $\sigma_{J'} \subset \gr^{k+1}_\otimes W_{\mathbf{k}, \mathbf{k}+1}$, we deduce, from Proposition \ref{prop:Wdecomp} and the description of $\gr^{k+1}_\otimes W_{\mathbf{k}, \mathbf{k}+1}$, that $\mathbf{k}(J') \geq \mathbf{k}(J)$ and $|\mathbf{k}(J')| =|\mathbf{k}(J)|+1$; in particular $k_i(J') - k_i(J) = \delta_{ij}$.
So if $i \neq j$, then $J\cap \{\pm \varepsilon_i\} = J'\cap \{\pm \varepsilon_i\}$.
While if $i=j$, then $J'\cap \{\pm \varepsilon_j\} = J\cap \{\pm \varepsilon_j\} \sqcup \{ \omega_J' - \omega_J\}$.
Hence $J'$ is of the above form.
\end{proof}

Fix $J\in \mathfrak{J}$.
Recall that by Proposition \ref{prop:Wdecomp}, there is a unique submodule of $W_{\mathbf{k}(J)}\subset R_\mu/\Fil^{>\mathbf{k}(J)} R_\mu$ isomorphic to $\sigma_J$.
If $P_{\sigma_J}$ is a projective envelope of $\sigma_J$, then the morphism $P_{\sigma_J} \surj \sigma_J \subset W_{\mathbf{k}(J)} \subset R_\mu/\Fil^{>\mathbf{k}(J)}_\otimes R_\mu$ lifts to a map $\psi_J: P_{\sigma_J} \rarrow R_\mu$.
We let $V_J$ be the image of $\psi_J$.
The following proposition partially describes the graded pieces of $V_J$.

\begin{prop}\label{prop:Vdecomp}
Let $J\in \mathfrak{J}$.
The filtration $\Fil^{\mathbf{k}}$ on $R_\mu$ induces a filtration on the submodule $V_J$.
Then for all $J'$ such that $J \subset J'$, $\sigma_{J'}\subset \gr^\mathbf{k(J')} V_J$.
\end{prop}
\begin{proof}
We proceed by induction on $k=k(J')$.
Suppose that $k < k(J)$.
Then $J \not\subset J'$, and there is no $J'$ as in the statement of the theorem. Thus the theorem holds in this case.

If $k = k(J)$, then $J\subset J'$ implies that $J' = J$.
By construction, $\sigma_J \subset \gr^{\mathbf{k}(J)} V_J$, and so the theorem holds in this case.

Now assume that $k > k(J)$ and that the theorem holds for $\gr^{k-1}_\otimes V_J$.
Suppose that $J'\in \mathfrak{J}$ such that $J\subset J'$ and $k(J') = k$.
Then there exists a $J''\in \mathfrak{J}$ such that $J\subseteq J'' \subset J'$ and $\#J'' = k-1$.
By the inductive hypothesis $\sigma_{J''} \subset \gr^{\mathbf{k}(J'')} V_J \subset \gr^{k-1}_\otimes V_J$.
We thus obtain a nonzero map $P_{\sigma_{J''}} \ra \Fil^{k-1}_\otimes V_J/\Fil^{k+1}_\otimes V_J$ which lifts the map $P_{\sigma_{J''}} \surj \sigma_{J''} \subset W_{\mathbf{k}(J'')}$, and therefore must be $\psibar_{J''}$.
By definition, the image of $\psibar_{J''}$ is $\overline{V}_{J''}$.
By Proposition \ref{prop:Vbardecomp}, $\sigma_{J'} \subset \gr^{\mathbf{k}(J')} \overline{V}_{J''} \subset \gr^{\mathbf{k}(J')} V_J$.
\end{proof}

For Proposition \ref{prop:ind}, we need the following two formal lemmas about tensor products of filtered vector spaces.

\begin{lem}\label{lem:linalg}
Let $\mathbf{k}$ and $\mathbf{k}' \in \{0,1,2\}^f$. Then $\Fil^{\mathbf{k}} R_\mu\cap \Fil^{\mathbf{k}'} R_\mu = \Fil^{\mathbf{k}''} R_\mu$ where $k''_i = \max(k_i,k'_i)$.
\end{lem}
\begin{proof}
Clearly, $\Fil^{\mathbf{k}''} R_\mu\subset \Fil^{\mathbf{k}} R_\mu\cap \Fil^{\mathbf{k}'} R_\mu$.
For each $i\in \Z/f$, choose a basis for $R_{\mu_i}$ compatible with the filtration and consider the corresponding tensor basis for $R_\mu$.
Then the elements of the tensor basis in $\Fil^{\mathbf{k}} R_\mu$ (resp.~$\Fil^{\mathbf{k}'} R_\mu$) form a basis for $\Fil^{\mathbf{k}} R_\mu$ (resp.~$\Fil^{\mathbf{k}'} R_\mu$).
Thus the elements of the tensor basis in $\Fil^{\mathbf{k}} R_\mu\cap \Fil^{\mathbf{k}'} R_\mu$ form a basis for $\Fil^{\mathbf{k}} R_\mu\cap\Fil^{\mathbf{k}'} R_\mu$.
These elements are in $\Fil^{\mathbf{k}''} R_\mu$, and so $\Fil^{\mathbf{k}} R_\mu\cap \Fil^{\mathbf{k}'} R_\mu \subset \Fil^{\mathbf{k}''} R_\mu$.
\end{proof}

For $I \subseteq \{0,1,2\}^f$, let $\Fil^I R_\mu := \sum_{\mathbf{k}\in I} \Fil^{\mathbf{k}} R_\mu$.

\begin{lem}\label{lem:intersect}
Let $I$ and $I' \subseteq \{0,1,2\}^f$.
Then 
\[\Fil^I R_\mu \cap \Fil^{I'} R_\mu = \sum_{\mathbf{k}\in I,\mathbf{k}'\in I'} \Fil^{\mathbf{k}} R_\mu\cap \Fil^{\mathbf{k}'} R_\mu.\]
\end{lem}
\begin{proof}
Clearly, $\sum_{\mathbf{k}\in I,\mathbf{k}'\in I'} \Fil^{\mathbf{k}} R_\mu\cap \Fil^{\mathbf{k}'} R_\mu \subset \Fil^I R_\mu \cap \Fil^{I'} R_\mu$.
For each $i\in \Z/f$, choose a basis for $R_{\mu_i}$ compatible with the filtration and consider the corresponding tensor basis for $R_\mu$.
Since the elements of the tensor basis in $\Fil^{\mathbf{k}} R_\mu$ span $\Fil^{\mathbf{k}} R_\mu$ for any $\mathbf{k}$, the elements of the tensor basis in $\Fil^I R_\mu$ (resp. $\Fil^{I'} R_\mu$) span $\Fil^I R_\mu$ (resp. $\Fil^{I'} R_\mu$) and thus form a basis for $\Fil^I R_\mu$ (resp. $\Fil^{I'} R_\mu$).
Thus the elements of the tensor basis in $\Fil^I R_\mu\cap \Fil^{I'} R_\mu$ form a basis for $\Fil^I R_\mu \cap \Fil^{I'} R_\mu$.
It is easy to see that a basis element is in $\Fil^I R_\mu$ (resp. $\Fil^{I'} R_\mu$) if and only it is in $\Fil^{\mathbf{k}} R_\mu$ (resp. $\Fil^{\mathbf{k}'} R_\mu$) for some $\mathbf{k}\in I$ (resp. $\mathbf{k}' \in I'$).
Thus $\Fil^I R_\mu \cap \Fil^{I'} R_\mu \subset \sum_{\mathbf{k}\in I,\mathbf{k}'\in I'} \Fil^{\mathbf{k}} R_\mu\cap \Fil^{\mathbf{k}'} R_\mu$.
\end{proof}

The following proposition shows that $V_J$ does not depend on the choice of lift $\psi_J$, but rather just on $J \in\mathfrak{J}$.

\begin{prop}\label{prop:ind}
Let $J\in \mathfrak{J}$.
Let $\psi'_J$ be a lift of the map $P_{\sigma_J}\surj \sigma_J \subset W_{\mathbf{k}(J)} \subset R_\mu/\Fil^{>\mathbf{k}(J)}R_\mu$.
Then the image of $\psi'_J$ lies in $V_J$.
In other words, $V_J$ does not depend on the choice of $\psi_J$.
\end{prop}
\begin{proof}
We recursively define maps $\phi^k:P_{\sigma_J} \rarrow \Fil^k_\otimes R_\mu \cap \Fil^{>\mathbf{k}(J)}R_\mu$ and maps $\psi^k:P_{\sigma_J} \ra \Fil^k_\otimes V_J$ for $k > k(J)$.
Let $\phi^{k(J)+1} = \psi'_J - \psi_J:P_{\sigma_J} \rarrow \Fil^{k(J)}_\otimes R_\mu \cap \Fil^{>\mathbf{k}(J)}R_\mu$.
Since $\psi'_J$ and $\psi_J$ coincide modulo $\Fil^{k(J)+1}_\otimes R_\mu \cap \Fil^{>\mathbf{k}(J)}R_\mu$, we see that the image of $\phi^{k(J)+1}$ lies in  $\Fil^{k(J)+1}_\otimes R_\mu \cap \Fil^{>\mathbf{k}(J)}R_\mu$.

We now define $\phi^{k+1}$ and $\psi^k$ in terms of $\phi^k$.
We first claim that the $\sigma_J$-isotypic part of $\gr^k_\otimes \Fil^{>\mathbf{k}(J)}R_\mu$ lies in $\gr^k_\otimes V_J$ for all $k$.
Indeed, by Lemmas \ref{lem:linalg} and \ref{lem:intersect}, $\Fil^k_\otimes \Fil^{>\mathbf{k}(J)}R_\mu$ (resp. $\Fil^{k+1}_\otimes \Fil^{>\mathbf{k}(J)}R_\mu$) is the sum $\sum_{\mathbf{k} > \mathbf{k}(J), |\mathbf{k}| \geq k} \Fil^{\mathbf{k}} R_\mu$ (resp. $\sum_{\mathbf{k} > \mathbf{k}(J), |\mathbf{k}| \geq k+1} \Fil^{\mathbf{k}} R_\mu$).
From this, we see that $\gr^k_\otimes \Fil^{>\mathbf{k}(J)}R_\mu = \oplus_{\mathbf{k} > \mathbf{k}(J), |\mathbf{k}|=k} W_{\mathbf{k}}$, which is $\oplus_{J'} \sigma_{J'}$ where the sum runs over $J'$ such that $\mathbf{k}(J') > \mathbf{k}(J)$ and $k(J') = k$ by Proposition \ref{prop:Wdecomp}.
If additionally $\sigma_{J'} \cong \sigma_J$, then $\omega_{J'} = \omega_J$ by Proposition \ref{prop:inj}.
The properties $\mathbf{k}(J') > \mathbf{k}(J)$ and $\omega_{J'} = \omega_J$ imply that for each $i\in \Z/f$, either $J'\cap \{\pm\varepsilon_i\} = J\cap\{\pm\varepsilon_i\}$ or $J\cap\{\pm\varepsilon_i\}$ is empty.
In any case, $J \subset J'$.
We conclude that $\sigma_{J'} \subset \gr^k_\otimes V_J$ by Proposition \ref{prop:Vdecomp}.

Thus the image of $\phi^k$ in $\gr^k_\otimes \Fil^{>\mathbf{k}(J)}R_\mu$, which is $\sigma_J$-isotypic, lies in $\gr^k_\otimes V_J$.
Let $\psi^k:P_{\sigma_J} \ra \Fil^k_\otimes V_J$ be a lift of the map $P_{\sigma_J} \ra \gr^k_\otimes V_J$ induced by $\phi^k$.
Let $\phi^{k+1} = \phi^k - \psi^k: P_{\sigma_J} \ra \Fil^k_\otimes R_\mu \cap \Fil^{>\mathbf{k}(J)}R_\mu$.
Since $\phi^k$ and $\psi^k$ coincide modulo $\Fil^{k+1}_\otimes R_\mu \cap \Fil^{>\mathbf{k}(J)}R_\mu$, the image of $\phi^{k+1}$ lies in $\Fil^{k+1}_\otimes R_\mu \cap \Fil^{>\mathbf{k}(J)}R_\mu$.

Then by construction, $\psi'_J = \psi_J + \sum_{k = k(J)+1}^{2f} \psi^k$.
Thus $\im \psi'_J \subset \im \psi_J + \sum_{k = k(J)+1}^{2f} \im \psi^k \subset V_J$.
\end{proof}

The following is the main submodule structure theorem for generic $\Gamma$-projective envelopes.

\begin{thm} \label{thm:structure}
Let $\mu\in X^{*}(T)$. Assume that $\mu-\eta$ is $1$-deep. Let $J'$ and $J\in \mathfrak{J}$ and let $V_{J'}$ and $V_J$ be the submodules of $R_{\mu}$ defined above Proposition \ref{prop:Vdecomp}.
If $J \subset J'$ then $V_{J'} \subset V_J$.
\end{thm}
\begin{proof}
Suppose that $J \subset J'$.
First note that $\sigma_{J'} \subset R_\mu/\Fil^{>\mathbf{k}(J')}R_\mu$ is contained in $V_J/\Fil^{>\mathbf{k}(J')}V_J$ by Proposition \ref{prop:Vdecomp}.
Let $\psi'_{J'}:P_{\sigma_{J'}} \ra V_J$ be a lift of the composition $P_{\sigma_{J'}} \surj \sigma_{J'} \subset V_J/\Fil^{>\mathbf{k}(J')}V_J \subset R_\mu/\Fil^{>\mathbf{k}(J')}R_\mu$.
Then $\im \psi'_{J'} = V_{J'}$ by Proposition \ref{prop:ind}.
We conclude that $V_{J'} \subset V_J$.
\end{proof}

Recall that for $J \in \mathfrak{J}$, we defined maps $\psi_J:P_{\sigma_J}\surj V_J \subset R_\mu$ above Proposition \ref{prop:Vdecomp}.
The following lemma will be useful for multiplicity computations.

\begin{lem} \label{lem:span}
Let $\sigma$ be a Serre weight and $P_\sigma$ a projective envelope of $\sigma$.
The vector space $\Hom_\Gamma(P_\sigma,R_\mu)$ is spanned by the set $\{\psi_J:\sigma_J \cong \sigma\}$.
\end{lem}
\begin{proof}
Since $P_\sigma$ is a projective $\Gamma$-module, $\Hom_\Gamma(P_\sigma,R_\mu) \isom \oplus_{\mathbf{k}} \Hom_\Gamma(P_\sigma, \gr^{\mathbf{k}} R_\mu)$.
Since $\gr^{\mathbf{k}} R_\mu$ is semisimple, $\Hom_\Gamma(P_\sigma, \gr^{\mathbf{k}} R_\mu) \isom \Hom_\Gamma(\sigma, \gr^{\mathbf{k}} R_\mu)$.
The space $\Hom_\Gamma(\sigma, \gr^{\mathbf{k}} R_\mu)$ is one-dimensional if there exists a $J\in \mathfrak{J}$ with $\mathbf{k}(J) = \mathbf{k}$ so that $\sigma \cong \sigma_J$ and is otherwise zero by Proposition \ref{prop:Wdecomp}.
In the case that $\Hom_\Gamma(\sigma, \gr^{\mathbf{k}} R_\mu)$ is nonzero, it is spanned by the image of $\psi_J$.
\end{proof}

\section{The Breuil--Pa\v{s}k\={u}nas construction}\label{sec:local}

In this section, we use the results of Section \ref{sec:proj} to give two distinct characterizations of a $\Gamma$-module constructed in \cite{BP}.

Let $F_v/\Q_p$ be an unramified extension.
Fix a tamely ramified representation $\rhobar:\G_{F_v} \rarrow \GL_2(\F)$, and let $R_w(\mu) = V_{\phi}(\rhobar^\vee(1))$ where $\mu = (\mu_i)_i\in X^*(T)$ and $w\in W=(S_2)^f$.

\begin{defin}\label{defin:generic}
We say that $\rhobar$ is {\it $1$-generic} if for all  possible choices of $\mu$ we have that $\mu-\eta$ is 1-deep in alcove $C_0$ and moreover the image of $\mu-\eta$ in the weight lattice of $G^{\der}$ is not of the form $\sum_{i=0}^{f-1}\varepsilon_i$ nor $\sum_{i=0}^{f-1}(p-3)\varepsilon_i$.
(We have $2^f$ possible choices for $\mu$, {\it a posteriori}.)

Concretely, if $\mu = \sum_i \mu_i^{(i)}$ where $\mu_i = (a_i,b_i)\in \Z^2$ then $\mu$ is $1$-generic iff $2\leq a_i - b_i \leq p-2$ for all $i$ and moreover $(a_i - b_i)_i\notin\{(2,\dots,2),\ (p-2,\dots,p-2)\}$.
\end{defin}
Note that if $\rhobar$ is $1$-generic
then for any $F(\mu-\eta) \in W^?(\rhobar^\vee(1))$  the corresponding projective envelope $R_{\mu}$ satisfies the  hypotheses of Theorem \ref{thm:structure}. Moreover, if $\rhobar$ is $1$-generic as in Definition \ref{defin:generic}, then it is in particular generic in the sense of \cite[Definition 11.7]{BP} and \cite[Definition 2.1.1]{EGS}.

We assume throughout that $\rhobar$ is $1$-generic. 
Let $\sigma := F(\mu-\eta) \in W^?(\rhobar^\vee(1))$.
Recall that the Weyl group $W$ acts naturally on $\Lambda_W$.
Let $S_w = w(S_e)$.
Then $W^?(\rhobar^\vee(1)) = F(\mathfrak{t}_\mu(\{\omega_J:J\subset S_w\}))$ by Proposition \ref{prop:weights}.
(We adopt the notation similar to (\ref{eqn:rep}): if $J\subset S_w$ define $\omega_J:=\sum_{\omega\in J}\omega$.)

\begin{defin}\label{defin:bp1wt}
Let $\rhobar$ be $1$-generic and let $\sigma := F(\mu-\eta) \in W^?(\rhobar^\vee(1))$.
We define the $\Gamma$-representation $D_0^\vee(\sigma,\rhobar)$ as
\begin{equation*}
D_0^\vee(\sigma,\rhobar) = R_\mu/\big(\sum_{\substack{J \subset S_w\\ \#J = 1}}V_J\big).
\end{equation*}
\end{defin}

\begin{lem} \label{lem:swmultone}
With the hypotheses of Definition \ref{defin:bp1wt}, the space
\[\Hom_\Gamma\bigg(\underset{\text{\tiny{$\kappa\in W^?(\rhobar^\vee(1))$}}}{\bigoplus} P_{\kappa},D_0^\vee(\sigma,\rhobar)\bigg)\] has dimension at most one and is nonzero if and only if $\kappa \cong \sigma$.
\end{lem}
\begin{proof}
Let $J_0\in \mathfrak{J}$ be such that $\sigma_{J_0}\cong\kappa\in W^?(\rhobar^\vee(1))$.
Recall from \S \ref{sec:proj} that for any $J\in \mathfrak{J}$ we have defined  a morphism $\psi_{J}: P_{\sigma_{J}} \rarrow R_\mu$ with image $V_{J}$.
By Lemma \ref{lem:span}, we see that the space $\Hom_\Gamma(P_{\sigma_{J_0}},R_\mu)$, and hence its quotient $\Hom_\Gamma(P_{\sigma_{J_0}},D_0^\vee(\sigma,\rhobar))$, is spanned by the image of $\{\psi_{J}:\omega_{J_0} = \omega_J\}$.
Thus it suffices to show that the image of $\psi_J$ in $\Hom_\Gamma(P_{\sigma_{J_0}},D_0^\vee(\sigma,\rhobar))$ is zero unless $J = \emptyset$ since $\sigma_\emptyset \cong \sigma$.

Let $J\in \mathfrak{J}$ such that $\omega_J = \omega_{J_0}$.
If $w\varepsilon_j\in J$ for some $j$, then $V_{J} \subset V_{\{w\varepsilon_j\}}$ by Theorem \ref{thm:structure},
and we conclude that the image of $\psi_J$ in 
\[\Hom_\Gamma(P_{\sigma_{J}},D_0^\vee(\sigma,\rhobar))\]
is $0$.
Thus if the image of $\psi_J$ is nonzero, then $w\varepsilon_j \notin J$ for all $j$.

If $w\varepsilon_j \notin J$ for all $j$, then $J \subset S_{w_0w}$ where $w_0\in W$ is the longest element.
Hence $\omega_J$ is in the closed $w_0w$-chamber in $X^*(T)$, and is $0$ if and only if $J = \emptyset$.
Since $\omega_J = \omega_{J_0}$ is also in the closed $w$-chamber in $X^*(T)$, we conclude that $\omega_J = 0$ and $J = \emptyset$.
Of course, the image of $\psi_\emptyset$ in $\Hom_\Gamma(P_{\sigma_{J}},D_0^\vee(\sigma,\rhobar))$ is nonzero.
\end{proof}
Let $D_0^\vee(\rhobar) = \oplus_{\sigma\in W^?(\rhobar^\vee(1))} D_0^\vee(\sigma,\rhobar)$.
Let $D_0(\rhobar)$ be $(D_0^\vee(\rhobar))^\vee$ (where $(\cdot)^\vee$ denotes the Pontrjagin duality).
The following proposition gives a characterization of $D_0^\vee(\rhobar)$, which is key for multiplicity one.

Recall that in \cite[Theorem 13.8]{BP} a $\Gamma$-representation $D_0(\rho)$ is attached to a \emph{generic} continuous Galois representation $\rho: G_{\Q_{p^f}}\rightarrow \GL_2(\F)$. For the sake of readibility, we denote this $\Gamma$-representation by $D_0^{\mathrm{BP}}(\rho)$.

\begin{prop}\label{prop:multone}
Assume that $\rhobar:\G_{F_v} \rarrow \GL_2(\F)$ is $1$-generic. Then $D_0(\rhobar)\cong D_0^{\mathrm{BP}}(\rhobar)$.
In particular the Jordan--H\"{o}lder factors of $D_0(\rhobar)$ appear with multiplicity one.
\begin{proof}
The cosocle of $D_0^\vee(\rhobar)$ is isomorphic to $\oplus_{\sigma \in W^?(\rhobar^\vee(1))} \sigma$, and for $\sigma \in W^?(\rhobar^\vee(1))$, $\sigma$ appears with multiplicity one in $D_0^\vee(\rhobar)$ by Lemma \ref{lem:swmultone}.
We will show that there is a surjection from $D_0^\vee(\rhobar)$ to any representation with these properties.

Indeed, assume that $Q$ is any $\Gamma$-representation with cosocle $\oplus_{\sigma \in W^?(\rhobar^\vee(1))} \sigma$ and such that any $\sigma \in W^?(\rhobar^\vee(1))$ appears with multiplicity one in $Q$.
Fix $\sigma \in W^?(\rhobar^\vee(1))$ and write $\sigma=F(\mu-\eta)$.
We have a map $R_\mu\rightarrow Q$ whose composite with $Q\surj \mathrm{cosoc}(Q)$ is non-zero. Let $J\in\mathfrak{J}$ be such that $J\subset S_w$ and $\#J=1$ (we follow the notations as in the beginning of this section) and write $Q_J$
for the image of $V_J\subseteq R_\mu$ in $Q$.
For any $J$ as above, if $Q_J=0$ then $V_J\subset \ker (R_\mu\rightarrow Q)$.
If $Q_J = 0$ for all $J$ as above, then the map $R_\mu \rarrow Q$ would factor through $D_0^\vee(\sigma,\rhobar)$.
If for all $\sigma \in W^?(\rhobar^\vee(1))$, $Q_J=0$ for all $J$ as above, then we would obtain a surjection $D_0^\vee(\rhobar)\surj Q$.
Assume for the sake of contradiction that for some $\sigma$ and some $J$ as above, $Q_J\neq 0$.
Then the modular weight $\sigma_J$ would appear as a Jordan--H\"older factor of the radical of $Q$.
However, $\sigma_J$ is also a Jordan--H\"older factor of the cosocle of $Q$, contradicting the multiplicity one assumption.

To conclude, note that $\sigma \in W^?(\rhobar^\vee(1))$ if and only if $\sigma^\vee \in W^?(\rhobar)$ (cf. e.g. \cite[Proposition 6.23]{Herzig}).
Hence by duality, $D_0(\rhobar)$ satisfies hypothesis \cite[Theorem 13.8(iii)]{BP}.
\end{proof}
\end{prop}

Recall that we denote by $W(\F)$ the ring of Witt vectors of $\F$.
If $\sigma(\tau)$ is a tame type defined over $W(\F)[1/p]$ and $\sigma^0(\tau)\subseteq\sigma(\tau)$ is a $\Gamma$-stable $W(\F)$-lattice in it, we denote by $\overline{\sigma}^0(\tau)$ the mod $p$ reduction of $\sigma^0(\tau)$.

\begin{lem}\label{lem:multone}
Suppose that $D_0^\vee$ is a $\Gamma$-representation such that $\dim \Hom_K(D_0^\vee,\sigma)$ is $1$ if $\sigma \in W^?(\rhobar^\vee(1))$ and $0$ otherwise.
Assume moreover that for any tame type $\sigma(\tau)$ and for any $W(\F)$-lattice $\sigma^0(\tau)\subseteq\sigma(\tau)$ such that $\mathrm{soc}(\overline{\sigma}^0(\tau))$ is irreducible, one has 
\[\dim \Hom_K(D_0^\vee,\overline{\sigma}^0(\tau)) \leq 1.\]
Then $\JH(\rad (D_0^\vee)) \cap W^?(\rhobar^\vee(1)) = \emptyset$.
\end{lem}
\begin{proof}
Suppose that $\sigma \in W^?(\rhobar^\vee(1))$, and $\sigma$ is a Jordan--H\"older factor of the radical of $D_0^\vee$.
By properties of projective envelopes, we can choose a $\Gamma$-surjection $\oplus_{\kappa\in W^?(\rhobar^\vee(1))} P_\kappa \surj D_0^\vee$. Let $I_\kappa\subset D_0^\vee$ the image of $P_\kappa$. We have $\rad (D_0^\vee)=\sum_{\kappa\in W^?(\rhobar^\vee(1))}\rad (I_\kappa)$, thus there is some $\kappa$ such that $\sigma$ is a Jordan--H\"older factor of the radical of $I_\kappa$.
For a $\Gamma$-representation $M$, we now denote by $\Fil^k M$ the cosocle filtration on $M$.
Then we have $\sigma \subset \gr^k D_0^\vee$ for some $k>0$.
Without loss of generality, suppose that $k$ is minimal among such Serre weights $\sigma\in W^?(\rhobar^\vee(1))$.

We claim that $k=1$.
Assume that $k>1$.

By minimality of $k$, $\Fil^1 I_\kappa/\Fil^k I_\kappa$ does not contain any weight in $W^?(\rhobar^\vee(1))$ as a Jordan--H\"older factor.
Let $\kappa = F(\mu-\eta)$ so that $P_\kappa \cong R_\mu$.
Thus that $V_J \subset \ker(\theta)$ for all $J$ such that $\#J = 1$ and $\sigma_J \in W^?(\rhobar^\vee(1))$.
By Lemma \ref{lem:swmultone}, $\rad(I_\kappa)$ does not contain any weight in $W^?(\rhobar^\vee(1))$ as a Jordan--H\"older factor, and in particular $\sigma$.
This is a contradiction.

Thus, there is a quotient $E$ of $D_0^\vee$ which has Loewy length two and socle isomorphic to $\sigma$.
Choose a type $\sigma(\tau)$ so that $\overline{\sigma}(\tau)$ contains $\kappa$ and $\sigma$ as Jordan--H\"older factors (one can even choose $\overline{\sigma}(\tau)$ to have Jordan--H\"older factors exactly the set $W^?(\rhobar^\vee(1))$).
There exists a unique up to homothety lattice $\sigma^0(\tau)$ such that $\mathrm{soc}(\overline{\sigma}^0(\tau))\cong \sigma$ (see \cite[Proposition 4.1.1]{EGS}).
There is an injection $E\inj \overline{\sigma}^0(\tau)$ by \cite[Theorem 5.1.1]{EGS}.
Then the maps $D_0^\vee \surj E \inj \overline{\sigma}^0(\tau)$ and $D_0^\vee \surj \sigma \inj \overline{\sigma}^0(\tau)$ are linearly independent, so that $\dim \Hom_K(D_0^\vee,\overline{\sigma}^0(\tau)) > 1$, a contradiction.
\end{proof}

The following proposition is an alternative characterization of $D_0^\vee(\rhobar)$.

\begin{prop} \label{prop:upperbound}
Suppose that $D_0^\vee$ is a $\Gamma$-representation such that $\Hom_K(D_0^\vee,\sigma)$ has dimension $1$ if $\sigma \in W^?(\rhobar^\vee(1))$ and $0$ otherwise.
Assume moreover that for any tame type $\sigma(\tau)$ and for any $W(\F)$-lattice $\sigma^0(\tau)\subseteq\sigma(\tau)$ such that $\mathrm{soc}(\overline{\sigma}^0(\tau))$ is irreducible, one has 
\[\dim \Hom_K(D_0^\vee,\overline{\sigma}^0(\tau)) \leq 1.\]
Then there is a $\Gamma$-surjection $D_0^\vee(\rhobar) \surj D_0^\vee$.
\end{prop}
\begin{proof}
By properties of projective envelopes, there is a $\Gamma$-equivariant surjection $\oplus_{\kappa\in W^?(\rhobar^\vee(1))} P_\kappa \surj D_0^\vee$.
Fix $\sigma \in W^?(\rhobar^\vee(1))$, and let $\sigma = F(\mu-\eta)$.
Let $\theta: R_\mu \ra D_0^\vee$ be a restriction of the above surjection to one direct summand.
Fix $J$ such that $\#J = 1$ and $\sigma_J \in W^?(\rhobar^\vee(1))$.
By Lemma \ref{lem:multone}, $\sigma_J$ does not appear in the image of $\theta$.
Thus $V_J\subseteq R_\mu$ is in the kernel of $\theta$, and the above surjection factors through $D_0^\vee(\rhobar)$.
\end{proof}

\section{Global applications}\label{sec:global}

In this section, we deduce our main theorem on cohomology of Shimura curves at full congruence level.
We are going to follow closely \cite{BD}, \S 3.2, 3.5 and 3.6, and \cite{EGS}, \S 6.5.

Recall that $F$ is a totally real field where $p$ is unramified. We write $\Sigma_p$ (resp.  $\Sigma_{\infty}$) the set of places of $F$ above $p$ (resp. above $\infty$). We write $\bA_F$ to denote the ring of ad\`eles of $F$.
We fix a continuous Galois representation $\rbar: G_F\rightarrow \GL_2(\F)$ which satisfies the following conditions:

\begin{enumerate}[(i)]
	\item $\rbar$ is modular;
	\item $\rbar|_{G_{F(\zeta_p)}}$ is absolutely irreducible;
	\item if $p=5$ then the image of $\rbar(G_{F(\zeta_p)})$ in $\mathbf{PGL}_2(\F)$ is not isomorphic to $A_5$;
	\item $\rbar|_{G_{F_w}}$ is \emph{generic} (in the sense of \cite{EGS}, Definition 2.1.1) for all $w\in \Sigma_p$.
\end{enumerate}
We write $\Sigma_{\rbar}$ for the ramification set of $\rbar$.
We fix the continuous character  $\psi: G_{F}\rightarrow \F^{\times}$ defined by $\psi:= \omega\det \rbar$ and write $\teich{\psi}$ to denote its Teichm\"uller lift.

Let $D$ be a quaternion algebra with center $F$ and let $\Sigma_D$ be the set of places where $D$ ramifies. We assume that: 
\begin{enumerate}[$\circ$]
	\item $\# (\Sigma_\infty\setminus \Sigma_D)\leq 1$;	
	\item $\Sigma_p\cap \Sigma_D=\emptyset$.
\end{enumerate}

We define  $S:=\Sigma_p\cup(\Sigma_D\setminus\Sigma_\infty)\cup\Sigma_{\rbar}$.
We note that the condition $p>3$ (coming from the genericity assumption on $\rbar|_{G_{F_w}}$) guarantees the existence of a place $w_1\notin S$ such that:
\begin{enumerate}[$\circ$]
	\item $\N(w_1)\not\equiv 1$ modulo $p$;
	\item the ratio of the eigenvalues of $\rbar(\mathrm{Frob}_{w_1})$ is not in $\{1,\ \N(w_1)^{\pm1}\}$; and
	\item if $\ell$ is a prime such that $[F(\sqrt[\ell]{1}):F]\leq 2$, then $w_1\nmid \ell$
\end{enumerate}
(cf.~\cite{BD}, item (iv) in the proof of Lemma 3.6.2).
If $\ell$ is the unique prime number which is divisible by $w_1$, we then define $K_{w_1}\leq(\cO_D)_{w_1}^{\times}$ as the pro-$\ell$-Iwahori of $(\cO_D)_{w_1}^{\times}$.
The conditions on $w_1$ and $K_{w_1}$ guarantee that for any open compact subgroup $K^{w_1}\leq (D\otimes_{F}\bA_F^{\infty,w_1})^{\times}$, the subgroup $K_{w_1} K^{w_1}$ is sufficiently small in the sense of \cite{GK}, \S 2.1.2.

We define $K^S:=\prod'_{w\notin S}K_w$ where $K_w:=(\cO_D)_w^{\times}$ for all $w\notin S\cup\{w_1\}$. 
We now follow the procedure of \cite{EGS}, \S 6.5 to obtain a space of algebraic automorphic forms with minimal tame level. 
We fix once and for all a place $v\in \Sigma_p$ and assume moreover that
\begin{enumerate}[(i)]\setcounter{enumi}{4}
	\item for all $w\in \Sigma_D$, $\rbar|_{G_{F_w}}$ is non-scalar.
\end{enumerate}

Let $S'\subseteq \Sigma_p\cup\Sigma_D$ be the subset of finite places $w\in \Sigma_p\cup\Sigma_D$ such that $\rbar|_{G_{F_w}}$ is reducible.
Write $W(\F)$ for the ring of Witt vectors of $\F$.
Following \cite{EGS}, \S 6.5 (which is in turn based on \cite{BD}, \S 3.3 and the proof of Proposition 3.5.1 in \emph{loc. cit.}), we fix for each $w\in S\setminus\{v\}$ the following data (we refer to \cite{EGS} and \cite{BD} for their precise definitions):
\begin{enumerate}[(1)]
	
	\item \label{primo} if $\rbar|_{G_{F_w}}$ is irreducible, the maximal compact $K_w:=(\cO_D)_w^{\times}$, an inertial type $\tau_w:I_w\rightarrow \GL_2(W(\F))$ (as in \cite[Proposition 3.5.1]{EGS} if $w\in \Sigma_p$, as in \cite{BD}, Cas IV in \S 3.3 else), and a $W(\F)$-lattice $L_w\subseteq \sigma(\tau_w)$. 
	
\item \label{secondo} if $\rbar|_{G_{F_w}}$ is reducible and $w\notin S'$, a compact subgroup $K_w\leq(\cO_D)_w^{\times}$ and a free $W(\F)$-module $L_w$ with a locally constant action of $K_w$ 
(cf. also \cite{BD}, Cas III at \S 3.3, and Cas (ii) in the proof of Proposition 3.5.1);

	\item \label{terzo} if $w\in S'$, a compact subgroup $K_w\leq(\cO_D)_w^{\times}$, a free $W(\F)$-module $L_w$ with a locally constant action of $K_w$ and a scalar $\beta_w\in \F^{\times}$ 
(cf. also \cite{BD}, Cas I and II at \S 3.3 and Cas (iii) in the proof of Proposition 3.5.1).

\end{enumerate}
We further remark that the $K_w$-representation $L_w$ has been chosen so that the center $F_w\cap K_w$ acts on $L_w$ via $\teich{\psi}\circ\Art_{F_w}$.
We define $K^v_S:= \underset{w\in S\setminus\{v\}}{\prod}K_w$, $K^v:= K^v_S K^{S}$ and $\tld{V}^v:= \underset{w\in S\setminus\{v\}}{\bigotimes}L_w$, which is a $W(\F)$-module of finite type with a locally constant action of $K^v_S$, hence of $K^v$ by inflation.
Via $\teich{\psi}$ we can and do endow $\tld{V}^v$ with an action of  $K^v(\bA_F^{(\infty,v)})^{\times}$. 
We write $\tld{V}^v_{\teich{\psi}}$ to denote the resulting $K^v(\bA_F^{(\infty,v)})^{\times}$-representation.

Let $K:= K^vK_v$.
Let $\Rep_{\F}^{\psi}(K_v)$ be the category of $\F$-modules of finite type, endowed with an action of
$K_v:= (\cO_D)_v^{\times}\cong \GL_2(\cO_{F_v})$
and such that $K_v\cap F_v^{\times}$ acts via the character $\psi\circ \Art_{F_v}$.
In particular if $V_v\in \Rep_{\F}^{\psi}(K_v)$ then the finite $\F$-module $V:= \tld{V}^v_{\tld{\psi}}\otimes V_v$ is endowed with an action of $K(\bA_F^{(\infty,v)})^{\times}$ which extends naturally to an action of $K (\bA_F^{\infty})^{\times}$.
We write $V_\psi$ to denote the resulting $K (\bA_F^{\infty})^{\times}$-representations.
By construction $(\bA_F^{\infty})^{\times}$ acts on  $V_\psi$  via $\psi$.

If $\# \big(\Sigma_\infty\setminus \Sigma_D\big)= 1$ we define the space of algebraic modular forms of level $K$, coefficients in $V_\psi$ and central character $\psi$ as:
\begin{equation}
\label{esp:faaIN}
S_{\psi}(K, V_v^{\vee}):= H^1_{\mathrm{\acute{e}t}}(X_K\otimes_F \overline{F},\mathcal{F}_{(V_\psi)^{\vee}})
\end{equation}
where $X_K$ is the smooth projective algebraic curve associated to $K$  as in \cite[\S 3.1]{BD} and $\mathcal{F}_{V_\psi^\vee}$ is the local system on $X_K\otimes_F \overline{F}$ associated to $V_\psi^\vee$ in the usual way (cf. \cite[proof of Lemma 6.2]{BD}).

If $\#\big( \Sigma_\infty\setminus \Sigma_D\big)= 0$ we define the space of algebraic modular forms of level $K$, coefficients in $V_\psi$ and central character $\psi$ as:
\begin{equation}
\label{esp:faa}
S_{\psi}(K, V_v^{\vee}):= \left\{\begin{matrix}
f: D^{\times}\backslash (D\otimes_F\bA_F^{\infty})^{\times}\rightarrow V_\psi^{\vee},\ f\ \text{continuous,}\\ \qquad f(gk)=k^{-1}f(g)\,\,\forall g\in (D\otimes_F\bA_F^{\infty})^{\times}, k\in K(\bA_F^{\infty})^{\times}
\end{matrix}
\right\}.
\end{equation}

We have a variation of the previous spaces with ``infinite level at $v$'' defined as follows:
\begin{equation*}
S_{\psi}(K^v, \F):= \underset{\substack{\rightarrow\\ U_v\leq K_v}}{\lim}S_{\psi}(K^vU_v, \F)
\end{equation*}
where $U_v$ ranges among the compact open subgroups $K_v$.
It is endowed with a smooth action of $D_v^{\times}\cong \GL_2(F_v)$.

The $\F$-modules $S_{\psi}(K, V_v^\vee)$, $S_{\psi}(K^v, \F)$ are faithful modules over a certain Hecke algebra which is defined as follows. 
Consider the $\F$-polynomial algebra $\bT^{S\cup\{w_1\}}:= \F[T_w^{(i)},\ w\notin S\cup\{w_1\}]$. For all $w\notin S\cup\{w_1\}$, $1\leq i\leq 2$ define the Hecke operator $T_w^{(i)}$ as the usual double classe operator acting on $S_{\psi}(K,V_v^{\vee})$:
$$
\left[ \GL_2(\cO_{F_w}) \left(\begin{matrix}
      \varpi_{w}\mathrm{Id}_i &  \cr  & \mathrm{Id}_{2-i} \end{matrix} \right)
\GL_2(\cO_{F_w}) \right]
$$
We then have an evident morphism of $\F$-algebras $\bT^{S\cup\{w_1\}}\rightarrow\End_W(S_{\psi}(K, V_v^{\vee}))$ whose image will be denoted by $\bT(V_v)$. From the hypothesis (i) there is a surjection $\alpha_{\rbar}: \bT(V_v)\rightarrow \F$ such that
\begin{equation*}
\det\big(X\mathrm{Id}_2-\overline{\psi}\rbar(\Frob_w)\big)=X-\alpha_{\rbar}(T_w^{(1)})X+\N(w)\alpha_{\rbar}(T_w^{(2)})
\end{equation*}
for all $w\notin S\cup\{w_1\}$. We note $\m_{\rbar}:= \ker(\alpha_{\rbar})$. 


For $w\in S'\cup\{w_1\}$ we can define the Hecke operator $T_w^{(1)}$ acting on $S_{\psi}(K, V_v^\vee)_{\m_{\rbar}}$  (cf. \cite{EGS} \S 6.5,
cf. also \cite{BD}, \S 3.3 Cas I et II), as well as scalars $\beta_w\in \F^{\times}$.  We write $\bT'(V_v)$  for the subalgebra of $\End_{\bT(V_v)}(S_{\psi}(K, V_v^\vee)_{\m_{\rbar}})$ generated by $\bT(V_v)$ and the operators $T_w^{(1)}$, $w\in S'\cup\{w_1\}$. 
In particular $\bT(V_v)_{\m_{\rbar}}\subseteq \bT'(V_v)$ is a finite extension of semi-local rings.
If $\m'_{\rbar}$ denotes the ideal of $\bT'(V_v)$ above $\m_{\rbar}$ and generated by the elements $T_w^{(1)}-\beta_w$, we easily see that $\m'_{\rbar}$  is a maximal ideal in $\bT'(V_v)$.

Note that the choices of types $\sigma(\tau_w)$, lattices $L_w\subseteq \sigma(\tau_w)$ and scalars $\beta_w$ (cf. items (\ref{primo}), (\ref{secondo}) and (\ref{terzo})  above) are exactly those of \cite[\S 6.5]{EGS} (in turn based on \cite[\S 3.3-3.5]{BD}) and the $\m'_{\rbar}$-generalized eigenspace of the modules (\ref{esp:faaIN}), (\ref{esp:faa}) are precisely the $\m'_{\rbar}$-generalized eigenspace of the spaces of fixed determinant algebraic modular forms with $V_\psi^{\vee}$-coefficients and minimal level as defined in \cite[\S 6.5]{EGS} (and denoted as $S^{\mathrm{min}}(V_v^{\vee})_{\m_{\rbar}}$ in \emph{loc.~cit.}).

\vspace{2mm}

If $\# \big(\Sigma_\infty\setminus \Sigma_D\big)= 1$ (resp. $\# \big(\Sigma_\infty\setminus \Sigma_D\big)= 0$) we define the smooth $K_v$-representation $\pi(\rhobar_v):=\Hom_{\F[G_F]}(\rbar, S_{\psi}(K^v, \F)[\m'_{\rbar}])$ (resp. $\pi(\rhobar_v):=S_{\psi}(K^v, \F)[\m'_{\rbar}]$) 
We set $K_v(1):=\ker(K_v\surj \Gamma)$.
From the main results in \cite{EGS} we have the following statement:
\begin{thm}[\cite{EGS}, Theorem 9.1.1 and 10.1.1]
\label{thmEGS}
Let $\rbar: G_F\rightarrow \GL_2(\F)$ be a continuous Galois representation satisfying the hypotheses (i)-(v) above.
Then 
\[
\mathrm{cosoc}_\Gamma((\pi(\rhobar_v)^{\vee})_{K_v(1)})=\underset{\sigma\in W^?(\rhobar_v(1)^{\vee})}{\oplus}\sigma.
\]
Let $\sigma(\tau)$ be a $K_v$-type and let $\sigma^{0}(\tau)$ a $W(\F)$-lattice with irreducible socle.
Then 
\[
\mathrm{Hom}_\Gamma((\pi(\rhobar_v)^\vee)_{K_v(1)},\overline{\sigma}^{0}(\tau))
\]
is at most one dimensional.
\end{thm}
\begin{proof}
We let $M_\infty: \Rep_{\F}^{\psi}(K_v)\rightarrow \mathrm{Mod}^{\mathrm{f.t.}}(R_\infty)$ be the fixed determinant and minimal level patching functor associated to $\rbar$ as in \cite[\S 6.5]{EGS}. 
By abuse of notation we let $\mathfrak{m}'_{\rbar}$ denote the maximal ideal of $R_\infty$.
By construction of the functor $M_{\infty}$, for any  representation $V_v\in \Rep_{\F}^{\psi}(K_v)$ we have an isomorphism
\[
\big(M_\infty(V_v)/\mathfrak{m}'_{\rbar}\big)^\vee\cong S_\psi(K^vK_v,V_v^\vee)[\mathfrak{m}'_{\rbar}]
\]
together with a compatible morphism of local rings $R^{\tld{\psi}}_\infty\rightarrow \bT'(V_v)_{\mathfrak{m}'_{\rbar}}$.

Since $K^vU_v$ is sufficiently small for any choice of a compact open subgroup $U_v\leq K_v$ and since $\mathfrak{m}'_{\rbar}$ is non-Eisenstein, a standard spectral sequence argument gives:
\[
\big(S_\psi(K^v,\F)[\mathfrak{m}'_{\rbar}]\big)^{K_v(1)}\cong S_\psi(K^vK_v(1),\F)[\mathfrak{m}'_{\rbar}].
\]

In particular if $K_v(1)$ acts trivially on $V_v\in \Rep_{\F}^{\psi}(K_v)$ we obtain
\begin{align}
\label{iso:patched}
\big(M_\infty(V_v)/\mathfrak{m}'_{\rbar}\big)^\vee&\cong S_\psi(K^vK_v,V_v^\vee)[\mathfrak{m}'_{\rbar}]
\\
\nonumber&\cong \Hom_{\Gamma}(V_v,S_\psi(K^vK_v(1),\F)[\mathfrak{m}'_{\rbar}])\\
\nonumber&\cong \Hom_{K_v}(V_v,\pi(\rhobar_v)^{K_v(1)}).
\end{align}

If $\sigma^0(\tau)$ is a lattice with irreducible  cosocle in a tame type $\sigma(\tau)$, we now deduce from \cite[Theorem 10.1.1]{EGS} that $\mathrm{Hom}_{K_v}(\overline{\sigma}^{0}(\tau),\pi(\rhobar_v))$ is at most one dimensional. 
With $\sigma^0(\tau)$ as in the statement of the theorem, $\overline{\sigma}^0(\tau)^\vee$ is the reduction of a lattice in the dual type $\overline{\sigma}(\tau)^\vee$ with irreducible cosocle 
and thus the second claim in the theorem follows by Pontrjagin duality.

By (\ref{iso:patched}), Nakayama's lemma, and Pontrjagin duality, $\sigma$ is a Jordan--H\"older factor of the $\Gamma$-cosocle of $(\pi(\rhobar_v)^{\vee})_{K_v(1)}$ if and only if $M_\infty(\sigma)\neq 0$.
By \cite[Theorem 9.1.1]{EGS}, $M_\infty(\sigma)\neq 0$ if and only if $\sigma\in W^?(\rhobar_v)$.
Finally, from the second part of the theorem, one sees that $\sigma$ appears in the $\Gamma$-cosocle of $(\pi(\rhobar_v)^{\vee})_{K_v(1)}$ with multiplicity one by taking any lattice in a tame type whose reduction has irreducible socle isomorphic to $\sigma$.
\end{proof}


From now on, we assume that:
\begin{enumerate}[(i)]\setcounter{enumi}{5}
 \item $\rhobar_v:= \rbar|_{G_{F_v}}$ is semisimple and $1$-generic in the sense of Definition \ref{defin:generic}.
\end{enumerate}


\begin{prop} \label{prop:breuil}
Let $\rbar: G_F\rightarrow \GL_2(\F)$ be a continuous Galois representation satisfying the hypotheses (i)-(vi) above. 
There is a $K_v$-surjection $\pi(\rhobar_v)^\vee \surj D_0^{\vee}(\rhobar_v)$.
\end{prop}
\begin{proof}
This is Pontrjagin dual to \cite[Proposition 9.3]{Breuil}, noting that $D_0^{\vee}(\rhobar_v)\cong (D_0^{\mathrm{BP}}(\rhobar_v))^\vee$.
\end{proof}

\begin{thm}\label{thm:main}
Let $\rbar: G_F\rightarrow \GL_2(\F)$ be a continuous Galois representation satisfying the hypotheses (i)-(vi) above. Then we have an isomorphism of $\Gamma$-modules
$(\pi(\rhobar_v)^\vee)_{K_v(1)} \cong D_0^\vee(\rhobar_v)$.
\end{thm}
\begin{proof}
By Proposition \ref{prop:breuil}, there is a surjection $(\pi^\vee)_{K_v(1)} \surj D_0^\vee(\rhobar_v)$.
By Theorem \ref{thmEGS}, 
$(\pi^\vee)_{K_v(1)}$ satisfies the conditions for $D_0^\vee$ in Proposition \ref{prop:upperbound}.
We conclude that there is a surjection $D_0^\vee(\rhobar_v) \surj (\pi^\vee)_{K_v(1)}$.
The composition of these surjections is a surjective endomorphism of $D_0^\vee(\rhobar_v)$, a finite length $\Gamma$-module, and is thus an isomorphism.
\end{proof}

We conclude with the main result of this paper:
\begin{corr}\label{corr:main}
Let $\rbar: G_F\rightarrow \GL_2(\F)$ be a continuous Galois representation satisfying the hypotheses (i)-(vi) above.
Then
\[
S_\psi(K^vK_v(1),\F)[\mathfrak{m}'_{\rbar}] \cong  D_0^{\mathrm{BP}}(\rhobar_v).
\]
In particular, the $\Gamma$-representation $S_\psi(K^vK_v(1),\F)[\mathfrak{m}'_{\rbar}]$ only depends on $\rbar|_{I_v}$ and is multiplicity free.
\end{corr}
\begin{proof}
%

Recall from the proof of Theorem \ref{thmEGS} the isomorphism:
\[
\big(S_\psi(K^v,\F)[\mathfrak{m}'_{\rbar}]\big)^{K_v(1)}\cong S_\psi(K^vK_v(1),\F)[\mathfrak{m}'_{\rbar}].
\]
The isomorphism follows now from Proposition \ref{prop:multone} and Theorem \ref{thm:main} after applying Pontrjagin duality.
For the second statement, recall that $D_0(\rhobar_v)$ was defined only in terms of $W^?(\rhobar_v^\vee(1))$ and is multiplicity free by Proposition \ref{prop:multone}.
\end{proof}

\bibliographystyle{amsalpha}
\bibliography{fullcongruence}

\providecommand{\bysame}{\leavevmode\hbox to3em{\hrulefill}\thinspace}
\providecommand{\MR}{\relax\ifhmode\unskip\space\fi MR }
\providecommand{\MRhref}[2]{%
  \href{http://www.ams.org/mathscinet-getitem?mr=#1}{#2}
}
\providecommand{\href}[2]{#2}
\begin{thebibliography}{LLHLM20}

\bibitem[AHHV17]{AHHV}
N.~Abe, G.~Henniart, F.~Herzig, and M.-F. Vign\'eras, \emph{A classification of
  irreducible admissible {${\rm mod}\, p$} representations of {$p$}-adic
  reductive groups}, J. Amer. Math. Soc. \textbf{30} (2017), no.~2, 495--559.
  \MR{3600042}

\bibitem[BD14]{BD}
Christophe Breuil and Fred Diamond, \emph{Formes modulaires de {H}ilbert modulo
  {$p$} et valeurs d'extensions entre caract\`eres galoisiens}, Ann. Sci. \'Ec.
  Norm. Sup\'er. (4) \textbf{47} (2014), no.~5, 905--974. \MR{3294620}

\bibitem[BDJ10]{BDJ}
Kevin Buzzard, Fred Diamond, and Frazer Jarvis, \emph{On {S}erre's conjecture
  for mod {$\ell$} {G}alois representations over totally real fields}, Duke
  Math. J. \textbf{155} (2010), no.~1, 105--161. \MR{2730374}

\bibitem[BP12]{BP}
Christophe Breuil and Vytautas Pa{\v{s}}k{\=u}nas, \emph{Towards a modulo {$p$}
  {L}anglands correspondence for {${\rm GL}_2$}}, Mem. Amer. Math. Soc.
  \textbf{216} (2012), no.~1016, vi+114. \MR{2931521}

\bibitem[Bre10]{BreuilICM}
Christophe Breuil, \emph{The emerging {$p$}-adic {L}anglands programme},
  Proceedings of the {I}nternational {C}ongress of {M}athematicians. {V}olume
  {II}, Hindustan Book Agency, New Delhi, 2010, pp.~203--230. \MR{2827792
  (2012k:22024)}

\bibitem[Bre14]{Breuil}
\bysame, \emph{Sur un probl\`eme de compatibilit\'e local-global modulo {$p$}
  pour {${\rm GL}_2$}}, J. Reine Angew. Math. \textbf{692} (2014), 1--76.
  \MR{3274546}

\bibitem[Col10]{Colmez}
Pierre Colmez, \emph{Repr\'esentations de {${\rm GL}_2(\bold Q_p)$} et
  {$(\phi,\Gamma)$}-modules}, Ast\'erisque (2010), no.~330, 281--509.
  \MR{2642409 (2011j:11224)}

\bibitem[EGS15]{EGS}
Matthew Emerton, Toby Gee, and David Savitt, \emph{Lattices in the cohomology
  of {S}himura curves}, Invent. Math. \textbf{200} (2015), no.~1, 1--96.
  \MR{3323575}

\bibitem[Eme11]{Emerton}
Matthew Emerton, \emph{Local-global compatibility in the $p$-adic langlands
  program for $\textrm{GL}_2/\mathbf{Q}$}, preprint (2011).

\bibitem[Gee11a]{Gee11}
Toby Gee, \emph{Automorphic lifts of prescribed types}, Math. Ann. \textbf{350}
  (2011), no.~1, 107--144. \MR{2785764}

\bibitem[Gee11b]{Gee}
\bysame, \emph{On the weights of mod {$p$} {H}ilbert modular forms}, Invent.
  Math. \textbf{184} (2011), no.~1, 1--46. \MR{2782251}

\bibitem[GHS18]{GHS}
Toby Gee, Florian Herzig, and David Savitt, \emph{General {S}erre weight
  conjectures}, J. Eur. Math. Soc. (JEMS) \textbf{20} (2018), no.~12,
  2859--2949. \MR{3871496}

\bibitem[GK14]{GK}
Toby Gee and Mark Kisin, \emph{The {B}reuil-{M}\'ezard conjecture for
  potentially {B}arsotti-{T}ate representations}, Forum Math. Pi \textbf{2}
  (2014), e1, 56. \MR{3292675}

\bibitem[Her09]{Herzig}
Florian Herzig, \emph{The weight in a {S}erre-type conjecture for tame
  {$n$}-dimensional {G}alois representations}, Duke Math. J. \textbf{149}
  (2009), no.~1, 37--116. \MR{2541127 (2010f:11083)}

\bibitem[Hu10]{Hu}
Yongquan Hu, \emph{Sur quelques repr\'esentations supersinguli\`eres de {${\rm
  GL}_2(\Bbb Q_{p^f})$}}, J. Algebra \textbf{324} (2010), no.~7, 1577--1615.
  \MR{2673752}

\bibitem[HW18]{Hu-Wang}
Yongquan Hu and Haoran Wang, \emph{Multiplicity one for the {$\mathrm{mod}\,
  p$} cohomology of {S}himura curves: the tame case}, Math. Res. Lett.
  \textbf{25} (2018), no.~3, 843--873. \MR{3847337}

\bibitem[Jan03]{JantzenBook}
Jens~Carsten Jantzen, \emph{Representations of algebraic groups}, second ed.,
  Mathematical Surveys and Monographs, vol. 107, American Mathematical Society,
  Providence, RI, 2003. \MR{2015057 (2004h:20061)}

\bibitem[LLHLM20]{LLLM2}
Daniel Le, Bao~Viet Le~Hung, Brandon Levin, and Stefano Morra, \emph{Serre
  weights and {B}reuil's lattice conjectures in dimension three}, Forum of
  Mathematics, Pi \textbf{8} (2020), e5.

\bibitem[Pa{\v{s}}13]{Paskunas}
Vytautas Pa{\v{s}}k{\=u}nas, \emph{The image of {C}olmez's {M}ontreal functor},
  Publ. Math. Inst. Hautes \'Etudes Sci. \textbf{118} (2013), 1--191.
  \MR{3150248}

\bibitem[Pil93]{Pil}
Cornelius Pillen, \emph{Reduction modulo {$p$} of some {D}eligne-{L}usztig
  characters}, Arch. Math. (Basel) \textbf{61} (1993), no.~5, 421--433.
  \MR{1241047}

\end{thebibliography}

\end{document}